\documentclass[12pt]{amsart}
\usepackage{amsmath, amsthm, amssymb}
\usepackage[english]{babel}
\usepackage{graphicx}
\usepackage{url}
\usepackage{caption}
\usepackage{float}
\usepackage{tikz-cd}
\usepackage{wrapfig}

\def\std{{\mathop{\fam0 i}}}
\def\delet{\mathaccent"7017 }
\def\hp{{h}}

\begin{document}

\title{The classification of certain linked $3$-manifolds in $6$-space}
\author{S. Avvakumov}
\thanks{I thank A. Skopenkov for telling me about the problem and for his useful remarks.
 I also thank A. Sossinsky, A. Zhubr, M. Skopenkov, P. Akhmetiev, and an anonymous referee for their feedback. Author was partially supported by Dobrushin fellowship, 2013, and by RFBR grant 15-01-06302.}
\email{s.avvakumov@gmail.com}
\maketitle

\newtheorem*{mainthmAbs}{Theorem}

\newtheorem{mainthm}{Theorem}
\newtheorem{mainthm1}[mainthm]{Theorem}
\newtheorem{Hthm}[mainthm]{Theorem}
\newtheorem*{hyp}{Hypothesis}

\theoremstyle{Lemma}
\newtheorem*{corAbs}{Corollary}

\newtheorem{cor1}{Corollary}

\newtheorem{lemEta}{Lemma}
\newtheorem{lemPT}[lemEta]{Lemma}
\newtheorem{lemL1hatwd}[lemEta]{Lemma}
\newtheorem{lemL1wd}[lemEta]{Lemma}
\newtheorem{lemL2wd}[lemEta]{Lemma}
\newtheorem{lemBasis}[lemEta]{Lemma}
\newtheorem{lemInBasis}[lemEta]{Lemma}
\newtheorem{lemTorIsotopy}[lemEta]{Lemma}
\newtheorem{lemSameParity}[lemEta]{Lemma}
\newtheorem{lemSameParity1}[lemEta]{Lemma}
\newtheorem{lemExistStandardized}[lemEta]{Lemma}
\newtheorem{lemIsotopic}[lemEta]{Lemma}
\newtheorem{lemChangeL1}[lemEta]{Lemma}
\newtheorem{lemChangeL2}[lemEta]{Lemma}
\newtheorem{lemSameLNumbers}[lemEta]{Lemma}
\newtheorem{lemFixedIsotopy}[lemEta]{Lemma}
\newtheorem{lemExamples}[lemEta]{Lemma}
\newtheorem{lemAdd}[lemEta]{Lemma}
\newtheorem{lemT1}[lemEta]{Lemma}
\newtheorem{lemHthmFixed}[lemEta]{Lemma}
\newtheorem{lemExistsDisk}[lemEta]{Lemma}
\newtheorem{lemLambda}[lemEta]{Lemma}
\newtheorem{lemDiskCobordism}[lemEta]{Lemma}
\newtheorem{lem2Isotopy}[lemEta]{Lemma}
\newtheorem{lemFramedSum}[lemEta]{Lemma}
\newtheorem{lemFramedHandlebodyIsotopy}[lemEta]{Lemma}

\theoremstyle{remark}
\newtheorem*{def1}{Definition of standard embeddings}
\newtheorem*{def2}{Construction of the embeddings $f_{k,0,0}$}
\newtheorem*{def3}{Definition of $\varkappa$}
\newtheorem*{def4}{Definition of a framed intersection}
\newtheorem*{def5}{Definition of a framed preimage}
\newtheorem*{def6}{Definition of standard framings and orientations}
\newtheorem*{def7}{Definition of $C_{\partial\std}$ and ${\delet f}$}
\newtheorem*{def8}{Definition of $\widehat{\nu}$ and $\nu$}
\newtheorem*{def10}{Definition of $\theta:{\rm Int}D^2_+\times S^1\rightarrow S^3$}
\newtheorem*{def9}{Definition of $\widehat{\mu}$}
\newtheorem*{def11}{Definition of $\mu$}
\newtheorem*{def12}{Definition of $\mu$ and $\nu$ for embeddings $S^3_1\sqcup S^3_2\rightarrow S^6$}
\newtheorem*{def13}{Construction of the embeddings $g_{m,n}$}
\newtheorem*{def14}{Definition of $\theta$}
\newtheorem*{def15}{Definition of $\rho$ and $\equiv$}
\newtheorem*{def16}{Notation for the Hopf invariant}
\newtheorem*{def17}{Notation for framed submanifolds}
\newtheorem*{def18}{Definition of a framed embedded connected sum}
\newtheorem*{def19}{Definition of a framed isotopy}
\newtheorem*{def20}{Notation for isotopies}
\newtheorem*{def21}{The natural orientation of a framed submanifold}
\newtheorem*{def22}{Orientation of intersection}
\newtheorem*{def23}{Agreement between orientations of manifolds and their boundaries}
\newtheorem*{def24}{Definition of the algebraic number of points of intersection}
\newtheorem*{def25}{Orientation of an embedded manifold}

\newtheorem{rem}{Remark}

\begin{abstract}
We work entirely in the smooth category. An embedding $f:(S^2\times S^1)\sqcup S^3\rightarrow {\mathbb R}^6$ is {\it Brunnian}, if the restriction of $f$ to each component is isotopic to the standard embedding. For each triple of integers $k,m,n$ such that $m\equiv n \pmod{2}$, we explicitly construct a Brunnian embedding $f_{k,m,n}:(S^2\times S^1)\sqcup S^3 \rightarrow {\mathbb R}^6$ such that the following theorem holds.

\begin{mainthmAbs}
\label{mainthmAbs}
Any Brunnian embedding $f:(S^2\times S^1)\sqcup S^3\rightarrow {\mathbb R}^6$ is isotopic to $f_{k,m,n}$ for some integers $k,m,n$ such that $m\equiv n \pmod{2}$. Two embeddings $f_{k,m,n}$ and $f_{k',m',n'}$ are isotopic if and only if $k=k'$, $m\equiv m' \pmod{2k}$ and $n\equiv n' \pmod{2k}$. 
\end{mainthmAbs}

We use Haefliger's classification of embeddings $S^3\sqcup S^3\rightarrow {\mathbb R}^6$ in our proof. The following corollary shows that the relation between the embeddings $(S^2\times S^1)\sqcup S^3\rightarrow {\mathbb R}^6$ and $S^3\sqcup S^3\rightarrow {\mathbb R}^6$ is not trivial.

\begin{corAbs}
\label{corAbs}
There exist embeddings $f:(S^2\times S^1)\sqcup S^3\rightarrow {\mathbb R}^6$ and $g,g':S^3\sqcup S^3\rightarrow {\mathbb R}^6$ such that the componentwise embedded connected sum $f\#g$ is isotopic to $f\#g'$ but $g$ is not isotopic to $g'$.
\end{corAbs}

\end{abstract}

\section{Introduction.}
\subsection*{Statement of the main result.}~\\
We work entirely in the smooth category.

In the statements of the results we use the following objects
explicitly constructed later in this section:
\begin{itemize}
\item{standard embeddings $S^3\rightarrow S^6$ and $S^2\times S^1\rightarrow S^6$;\footnote{For $m\geq n+2$ classifications of embeddings of $n$-manifolds into $S^m$ and into ${\mathbb R}^m$ are the same (see \cite[1.1 Sphere and Euclidean space]{Atl1}). It is more convenient to us to consider embeddings into $S^6$ instead of ${\mathbb R}^6$.}}
\item{embedding $f_{k,0,0}:S^2\times S^1\sqcup S^3 \rightarrow S^6$ for each integer $k$;\footnote{Here and below $S^2\times S^1\sqcup S^3$ means $(S^2\times S^1)\sqcup S^3$.}}
\item{embedding $g_{m,n}:S^3\sqcup S^3\rightarrow S^6$ for each pair $m,n$ of integers such that $m\equiv n \pmod{2}$;}
\item{componentwise embedded connected sum operation $\#$.}
\end{itemize}

An embedding $S^2\times S^1\sqcup S^3\rightarrow S^6$ or $S^3\sqcup S^3\rightarrow S^6$ is {\it Brunnian}, if its restriction to each component is isotopic to the standard embedding.

Embeddings $f_{k,0,0}:S^2\times S^1\sqcup S^3 \rightarrow S^6$ and $g_{m,n}:S^3\sqcup S^3\rightarrow S^6$ are Brunnian.

For each triple of integers $k,m,n$ such that $m\equiv n \pmod{2}$, let $f_{k,m,n}:=f_{k,0,0}\#g_{m,n}$. Each embedding $f_{k,m,n}$ is Brunnian because both $f_{k,0,0}$ and $g_{m,n}$ are Brunnian.

\begin{mainthm}
\label{mainthm}
Any Brunnian embedding $f:S^2\times S^1\sqcup S^3\rightarrow S^6$ is isotopic to $f_{k,m,n}$ for some integers $k,m,n$ such that $m\equiv n \pmod{2}$. Two embeddings $f_{k,m,n}$ and $f_{k',m',n'}$ are isotopic if and only if $k=k'$, $m\equiv m' \pmod{2k}$, and $n\equiv n' \pmod{2k}$.
\end{mainthm}

\begin{rem}
Theorem \ref{mainthm} implies that there is a surjective map from the set of isotopy classes of Brunnian embeddings $S^2\times S^1\sqcup S^3\rightarrow S^6$ to ${\mathbb Z}$; the preimage of any $k\in {\mathbb Z}$ under this map is in one-to-one correspondence with the set $\frac{\{(m,n)\in {\mathbb Z}\times {\mathbb Z}|m\equiv n \pmod{2}\}}{(m,n)\sim (m+2k,n) \sim (m,n+2k)}$.
\end{rem}

The invariants necessary for the proof of Theorem \ref{mainthm} are defined in Section 2.

The classification of Brunnian embeddings $S^3\sqcup S^3\rightarrow S^6$ is given by the following Haefliger Theorem (\cite[Example 1.2]{MSk09}, see also \cite[\S 6]{Ha62}, \cite[Theorem 10.7]{Ha66}).

\begin{Hthm}[Haefliger]
\label{Hthm}
Any Brunnian embedding $g:S^3\sqcup S^3\rightarrow S^6$ is isotopic to $g_{m,n}$ for some integers $m,n$ such that $m\equiv n \pmod{2}$. Two embeddings $g_{m,n}$ and $g_{m',n'}$ are isotopic if and only if $m=m'$ and $n=n'$.
\end{Hthm}

Theorem \ref{mainthm} together with Theorem \ref{Hthm} implies the following Corollary \ref{cor1}.

\begin{cor1}
\label{cor1}
There exist embeddings $f:S^2\times S^1\sqcup S^3\rightarrow S^6$ and $g,g':S^3\sqcup S^3\rightarrow S^6$  such that $f\#g$ is isotopic to $f\#g'$ but $g$ is not isotopic to $g'$.
\end{cor1}

Corollary \ref{cor1} shows that the relation between the embeddings $S^2\times S^1\sqcup S^3\rightarrow S^6$ and $S^3\sqcup S^3\rightarrow S^6$ is not trivial. The corollary is deduced from the theorems above near the end of this section.

We conjecture that the following analogue of Theorem \ref{mainthm} also holds in piecewise-smooth category. The classification of piecewise-smooth embeddings coincides with the classification of PL embeddings (\cite{Ha67}). So if the following hypothesis is true then the analogue of Theorem \ref{mainthm} holds in both piecewise-smooth and PL categories.

\begin{hyp}
Any Brunnian piecewise-smooth embedding $f:S^2\times S^1\sqcup S^3\rightarrow S^6$ is piecewise-smooth isotopic to $f_{k,m,n}$ for some integers $k,m,n$ such that $m\equiv n \pmod{2}$. Two embeddings $f_{k,m,n}$ and $f_{k',m',n'}$ are piecewise-smooth isotopic if and only if $k=k'$, $m\equiv m' \pmod{2k}$, and $n\equiv n' \pmod{2k}$.
\end{hyp}

\subsection*{Discussion of the main result.}~\\
By a {\it $p$-link in $q$-space} we mean an embedding $S^p\sqcup S^p\rightarrow {\mathbb R}^q$. A lot of interesting results on $1$-links in $3$-space were obtained during the rapid advancement of topology in the 20th century. However, it was soon realized that the complete isotopy classification of these links is unachievable even for links with unknotted components. The same is true for high-dimensional links in codimension $2$, i.e., when $q=p+2>3$.

The situation is less grim in codimension greater than $2$. If $q\geq p+3$ one can define the generalized linking numbers of a link, which are its isotopy invariants (see \cite[definition of $\lambda_{12}$]{Atl2}, or \cite[\S 4]{Ha62}, or \cite[definition of $\lambda$]{MSk09}, or definition of $\mu$ and $\nu$ for embeddings $S^3_1\sqcup S^3_2\rightarrow S^6$ in Section 4 of this text). 

The generalized linking numbers classify $p$-links in $q$-space 
\begin{itemize}
\item{up to h-cobordism, if $q>\frac{3p+3}{2}$ (\cite[\S 5]{Ha62});}
\item{up to h-cobordism modulo knotting of their components, if $q=\frac{3p+3}{2}, q > 9, q\neq 21$ (\cite[\S 6]{Ha62});}
\item{up to isotopy modulo knotting of their components, if $3q>4p+6$ (\cite[Theorem 10.7]{Ha66}).}
\end{itemize}

The latter statement also holds if $3q=4p+6$ (\cite[Theorem 1.1]{MSk09}). The smallest pair $(q,p)$ of numbers satisfying $3q\geq 4p+6$, $q\geq p+3$ and such that the classification of $p$-links in $q$-space is not trivial is $(6,3)$.

For some $q\leq \frac{3p+3}{2}$, the related problem of the {\it link homotopy} classification of linked $p$-{\it manifolds} in ${\mathbb R}^q$ was solved by A. Skopenkov \cite{Sk00} (when these $p$-manifolds are spheres there are stronger results, see bibliography in \cite{Sk00}). However, Theorem \ref{mainthm} is the first result on the isotopy classification of some linked $p$-manifolds (different from homology spheres) in ${\mathbb R}^q$ for $q\leq \frac{3p+3}{2}$.

To prove Theorem \ref{mainthm}, we define three isotopy invariants $\varkappa$, $\mu$, and $\nu$ which constitute a complete system of invariants for our problem (see Section 2). The definition of the last two invariants $\mu$ and $\nu$ is rather similar, albeit more complicated, to the definition of the generalized linking numbers. It is harder to prove, however, that $\mu$ and $\nu$ are well defined. The first invariant $\varkappa$ is very simple but has no analogue in the case of embeddings $S^3\sqcup S^3\rightarrow S^6$ (more precisely, its analogue is trivial).

We partly reduce the injectivity of the invariants to the Haefliger Theorem (Theorem \ref{Hthm}). Another important part of the proof of the injectivity (see Section 7, Lemmas \ref{lemChangeL1}, \ref{lemChangeL2}, and \ref{lem2Isotopy}) involves explicitly constructed isotopies between some ``good'' embeddings $(S^2\times S^1)\sqcup S^3\rightarrow {\mathbb R}^6$ (which we call {\it simple} and {\it upper simple} embeddings, see definition in Section 2).

The surjectivity of the invariants is proved by a rather simple explicit construction of the embeddings $f_{k,m,n}$ and by Lemma \ref{lemExamples}.

\subsection*{Embedded connected sum.}~\\
Let $M$ and $N$ be compact oriented connected manifolds, ${\rm dim}M={\rm dim}N=p$. Let $b_M:D^p\rightarrow M$ and $b_N:D^p\rightarrow N$ be orientation preserving and orientation reversing embeddings, respectively. Denote by $M_0$ and $N_0$ the closures of the complements of $M$ and $N$ to $b_M(D^p)$ and $b_N(D^p)$, respectively. Define the connected sum $M\#N$ as $$M_0\cup_{b_M|_{S^{p-1}}} S^{p-1}\times I\cup_{b_N|_{S^{p-1}}} N_0.$$

	\begin{wrapfigure}{r}{0.3\textwidth} 
	\begin{center} 
	\includegraphics[width=0.28\textwidth]{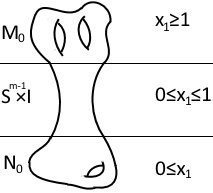} 
	\end{center} 
	\caption{Connected sum.}
	\label{figConnectedSum} 
	\end{wrapfigure}

The smooth structure on $M\#N$ is defined as follows. Denote by $x_1$ the first coordinate in ${\mathbb R}^{2p+2}$. By general position, there is an embedding $u:M\#N\rightarrow {\mathbb R}^{2p+2}$ such that $u(M_0)$, 
$u(N_0)$, and $u(S^{p-1}\times I)$ are proper smooth submanifolds of $\{x_1\leq 0\}$, $\{x_1\geq 1\}$, and $\{0\leq x_1\leq 1\}$, respectively (Fig.~\ref{figConnectedSum}). Then the image of $u$ is a smooth submanifold of ${\mathbb R}^{2p+2}$ and is equipped with a natural smooth structure. The smooth structure on $M\#N$ is then induced by $u$. By general position, any other embedding $u':M\#N\rightarrow {\mathbb R}^{2p+2}$ is isotopic to $u$. So the smooth structure on $M\#N$ is well defined.

Let $f:M\rightarrow S^q$ and $g:N\rightarrow S^q$ be embeddings with disjoint images. Take an embedding $l:D^p\times I\rightarrow S^q$, such that $l|_{D^p\times 0}=f\circ b_M$, $l|_{D^p\times 1}=g\circ b_N$, and such that ${\rm Im}l$ does not intersect elsewhere the union of $f(M)$ and $g(N)$.

The {\it embedded connected sum} of the embeddings $f$ and $g$ (along the tube $l$) is the embedding $f\#g:M\#N\rightarrow S^q$, defined by $f\#g|_{M_0}=f$, $f\#g|_{N_0}=g$, and $f\#g|_{S^{p-1}\times I}=l$ (compare with \cite[\S 2]{Ha62} and \cite[\S 3]{Ha66}). We may choose $l$ so that the map $f\#g$ is smooth.

Note that $\#$ is defined for embeddings, not isotopy classes.\footnote{One can also define the corresponding operation for isotopy classes $[f\sqcup g]$ by the formula $\#([f\sqcup g]):=[f\# g]$, where $f:M\rightarrow S^q$ and $g:N\rightarrow S^q$ are embeddings with disjoint images, and $[f]$ denotes the isotopy class of $f$, see (\cite[3 Embedded connected sum]{Atl1}). The isotopy class $\#([f\sqcup g])$ is well defined, i.e., does not depend on the choice of $l$ and the smoothing of $f\#g$, if and only if $q-p\geq 3$. The proof of this fact is not trivial. For example, it can be deduced from Standartization Lemma 2.1 in \cite{Sk15} by substituting $(0,p,q)$ for $(p,q,m)$.}
The result of $\#$ depends on the choice of $l$ which we will omit in the rest of the text.

\subsection*{Usage of the term ``tubular neighborhood''.}~\\
By a {\it tubular neighborhood} of a submanifold we mean an open subset in the ambient manifold (\cite[Chapter 4.5]{Hi}). Note that sometimes a tubular neighborhood is defined as an {\it embedding} of the total space of a vector bundle over a submanifold into the ambient manifold, not as the image of this embedding.

\subsection*{Definition of standard embeddings.}~\\
Let $x_1, x_2, \ldots, x_m$ be the coordinates in ${\mathbb R}^m$.

Denote by $S^{p-1}$ and $D^p$ the unit sphere and the unit disk in ${\mathbb R}^p$, respectively. For $p < m$ identify ${\mathbb R}^p$ with the hyperplane of ${\mathbb R}^m$ given by the equations $x_{p+1}=x_{p+2}=\ldots=x_m=0$. Thus $S^{p-1}$ and $D^p$ are identified with subsets of $S^{m-1}$ and $D^m$, respectively. The obtained inclusions are called the {\it standard embeddings} $S^{p-1} \rightarrow S^{m-1}$ and $D^p \rightarrow D^m$. 

The {\it standard embedding} $D^3 \times S^1 \rightarrow S^4$ 
is given by the formula 
$$({\bf x}, (y_1,y_2)) \mapsto ({\bf x}, y_1\sqrt{2-{\bf x}^2},y_2\sqrt{2-{\bf x}^2})/\sqrt{2}\in S^4,$$ where ${\bf x}$ denotes $(x_1,x_2,x_3)\in D^3$. The image of the embedding is the closure of a tubular neighborhood $\{x_1^2+x_2^2+x_3^2\leq\frac{1}{2}, x_4^2+x_5^2\geq\frac{1}{2}\}\subset S^4$ of the circle\footnote{From the equation $x_4^2+x_5^2=1$ it follows that $x_1=x_2=x_3=0$. To shorten the notation, we do not include the equation $x_1=x_2=x_3=0$ in the formula of the circle. The same goes for rest of the text.} $\{x_4^2+x_5^2=1\}\subset S^4$.

The {\it standard embedding} $\std:D^3 \times S^1 \rightarrow S^6$ is the composition of standard embeddings $D^3 \times S^1 \rightarrow S^4$ and $S^4 \rightarrow S^6$.

The {\it standard embedding} $S^2 \times S^1 \rightarrow S^6$ is the restriction $\std|_{S^2 \times S^1}$. The image of the embedding is $\{x_1^2+x_2^2+x_3^2=\frac{1}{2}, x_4^2+x_5^2=\frac{1}{2}\}\subset S^6$.

\subsection*{Construction of the embeddings $f_{k,0,0}$ and $g_{m,n}$. Proof of Corollary \ref{cor1}.}
\begin{def2}
Take any embedding $t:S^3\rightarrow S^6$ isotopic to the standard embedding and such that $t(S^3)$ lies in a $6$-disk disjoint from $\std(S^2\times S^1)$. Let $f_{0,0,0}:S^2\times S^1\sqcup S^3\rightarrow S^6$ be the embedding such that $f_{0,0,0}|_{S^2\times S^1}=\std|_{S^2\times S^1}$ and $f_{0,0,0}|_{S^3}=t$.

For $\frac{1}{10}\geq\epsilon\geq 0$ define the embeddings $z_{3,\epsilon}:S^3\rightarrow S^6$ and $\overline{z}_{3,\epsilon}:S^3\rightarrow S^6$ by the formulae:
$$z_{3,\epsilon}(x_1,x_2,x_3,x_4)=(\epsilon,0,0,(x_1,x_2,x_3,x_4)\sqrt{1-\epsilon^2})$$
$$\overline{z}_{3,\epsilon}(x_1,x_2,x_3,x_4)=(\epsilon,0,0,(x_1,x_2,x_3,-x_4)\sqrt{1-\epsilon^2})$$

The embeddings $z_{3,\epsilon}$ and $\overline{z}_{3,\epsilon}$ are isotopic to the standard embeddings. The spheres $z_{3,\epsilon}(S^3)$ and $\overline{z}_{3,\epsilon}(S^3)$ are linked with $\std(S^2\times \cdot)$. The absolute value of the linking numbers is $1$.

For an integer $k>0$, let $f_{k,0,0}:S^2\times S^1\sqcup S^3\rightarrow S^6$ be the embedding such that $$f_{k,0,0}|_{S^3}=z_{3,\frac{1}{10k}}\# z_{3,\frac{2}{10k}} \#\ldots\# z_{3,\frac{k}{10k}} \text{\quad and\quad} f_{k,0,0}|_{S^2\times S^1}=\std|_{S^2\times S^1}.$$

For an integer $k<0$, let $f_{k,0,0}:S^2\times S^1\sqcup S^3\rightarrow S^6$ be the embedding such that $$f_{k,0,0}|_{S^3}=\overline{z}_{3,\frac{1}{10|k|}}\# \overline{z}_{3,\frac{2}{10|k|}} \#\ldots\# \overline{z}_{3,\frac{|k|}{10|k|}} \text{\quad and\quad} f_{k,0,0}|_{S^2\times S^1}=\std|_{S^2\times S^1}.$$

The embedding $f_{k,0,0}$ is Brunnian for each $k$.
\end{def2}

\begin{def13}
Let $g_{0,0}:S^3\sqcup S^3\rightarrow S^6$ be some Brunnian embedding such that the image of the first component lies in an open $6$-ball disjoint from the image of the second component. We call $g_{0,0}$ a {\it trivial link}.

The Zeeman map $\tau:\pi_3(S^2)\rightarrow {\rm Emb}^6(S^3\sqcup S^3)$ is defined in \cite{Atl2}. Let $\eta:S^3\rightarrow S^2$ be the Hopf map. 

The Whitehead link $\omega:S^3\sqcup S^3\rightarrow S^6$ is also defined in \cite{Atl2}.

Both embeddings $\tau(\eta),\omega:S^3\sqcup S^3\rightarrow S^6$ are Brunnian\footnote{The image of $\tau(\eta)$ is as follows. Let $S^3\times D^3\subset S^6$ be the {\it standard tubular neighborhood} of the standard sphere $S^3\subset S^6$. Then the image of $\tau(\eta)$ is the union of the standard sphere $S^3\subset S^6$ and the graph of the Hopf map $\eta:S^3\rightarrow S^2$ in the boundary $S^3\times S^2\subset S^6$ of $S^3\times D^3\subset S^6$.}. For $m\equiv n \pmod{2}$, let $g_{m,n}:=m\tau(\eta)\#\frac{(n-m)}{2}\omega$. Here both $0\tau(\eta)$ and $0\omega$ are trivial links, i.e., some links isotopic to $g_{0,0}$.
\end{def13}

\begin{proof}[Proof of Corollary \ref{cor1}]
Consider the trivial link $g_{0,0}:S^3\sqcup S^3\rightarrow S^6$ and the Whitehead link $g_{0,2}=\omega:S^3\sqcup S^3\rightarrow S^6$.

By Theorem \ref{mainthm}, the embeddings $f_{1,0,0}\#g_{0,0}$ and $f_{1,0,0}\#\omega$ are isotopic, because $f_{1,0,0}\#g_{0,0}$ is isotopic to $f_{1,0,0}$ and $f_{1,0,0}\#\omega=f_{1,0,2}$. However, by the Haefliger Theorem, the Whitehead link $g_{0,2}=\omega$ is not trivial, i.e., is not isotopic to $g_{0,0}$.
\end{proof}

\subsection*{Agreement between the orientations of a manifold and its boundary, the natural orientation of a framed submanifold, framed intersection, and framed preimage.}

\begin{def23}
Let $M^m$ be an oriented manifold with a boundary. Let $u_1, \ldots, u_{m-1}$ be a positive tangent frame of $\partial M$ at $P\in\partial M$. Let $v$ be a vector at $P$ tangent to $M$ and ``looking'' outside $M$. Orientations of $M$ and $\partial M$ {\it agree} if the frame $v, u_1, \ldots, u_{m-1}$ is a positive tangent frame of $M$.
\end{def23}

In this text {\it framing} means {\it normal framing}.

\begin{def21}
Let $M\subset S^k$ be a framed submanifold. Let $u_1,\ldots, u_m$ be a tangent frame and let $v_1,\ldots,v_{k-m}$ be the vectors of the framing of $M$ at some point $P\in M$. The frame $u_1,\ldots, u_m$ is a {\it positive} tangent frame of $M$ at $P$ if $u_1,\ldots, u_m, v_1,\ldots,v_{k-m}$ is a positive tangent frame of $S^k$ at $P$. The defined orientation of $M$ is called its {\it natural} orientation.
\end{def21}

\begin{def25}
Let $M$ be an oriented manifold and let $f:M\rightarrow N$ be an embedding. Then we assume that $f(M)$ has an orientation induced by $f$.
\end{def25}

If we frame $f(M)$ we do it so that the natural orientation of $f(M)$ is the same as the orientation induced by $f$.

\begin{def22}
Let $M^m, N^n\subset S^k$ be transversal oriented submanifolds. Take a point $P\in M\cap N$. Let $w_1,\ldots,w_{m+n-k}$ be a tangent frame of $M\cap N$ at $P$. 
\\Let $w_1,\ldots,w_{m+n-k},u_1,\ldots,u_{k-n}$ be a positive tangent frame of $M$ at $P$. 
\\Let $w_1,\ldots,w_{m+n-k},v_1,\ldots,v_{k-m}$ be a positive tangent frame of $N$ at $P$.

The frame $w_1,\ldots,w_{m+n-k}$ is a {\it positive} tangent frame of $M\cap N$ at $P$ if the frame \\
$w_1,\ldots,w_{m+n-k},u_1,\ldots,u_{k-n},v_1,\ldots,v_{k-m}$ is a positive tangent frame of $S^k$ at $P$.
Note that $M\cap N=(-1)^{(k-m)(k-n)}(N\cap M)$.
\end{def22}

\begin{def4}
Let $M, N\subset S^k$ be transversal submanifolds. Assume that $M$ is framed. For each point $P\in M\cap N$ project the framing of $M$ at $P$ onto the tangent space of $N$ at $P$. Orthonormalise the obtained family of normal to $M\cap N$ vectors using the Gram-Schmidt process. The resulting framed submanifold $M\cap N$ of $N$ is the {\it framed intersection} of $M$ and $N$.
\end{def4}

In this text any intersection of a submanifold and a framed submanifold is assumed to be framed.

\begin{def24}
Let $M, N\subset S^k$ be transversal oriented submanifolds, ${\rm dim}(M)+{\rm dim}(N)=k$. Let $u_1,\ldots,u_m$ and $v_1,\ldots,v_{k-m}$ be positive tangent frames of $M$ and $N$, respectively, at $P\in M\cap N$. The point $P$ is called a {\it positive} ({\it negative}) point of the intersection $M\cap N$ if the frame $u_1,\ldots,u_m,v_1,\ldots,v_{k-m}$ is a positive (negative) tangent frame of $S^k$ at $P$. The {\it algebraic number of points of the intersection} $M\overset{\cdot}{\cap} N$ is the difference between the number of positive and negative points of $M\cap N$.
\end{def24}

\begin{def5}
Let $f:M\rightarrow N$ be an embedding and let $L\subset N$ be a framed submanifold transverse to $f(M)$. 
Then $L\cap f(M)$ is a framed submanifold of $f(M)$. The framed $f$-preimage of $L$ is $f^{-1}(L\cap f(M))$ and is denoted by $f^{-1}(L)$ for brevity.
\end{def5}

In this text any preimage of a framed submanifold is assumed to be framed.

\medskip
\section{Plan of the proof of the main result.}
In this section we prove Theorem \ref{mainthm} modulo Lemmas \ref{lemL2wd} through \ref{lemT1}. Proofs of all the Lemmas depend on this section. Lemmas \ref{lemExistStandardized} and \ref{lemExistsDisk} are proved in this section. Lemmas \ref{lemL1hatwd} and \ref{lemL1wd} are proved in Section 3. Lemma \ref{lemSameParity} is proved in Section 4. Lemma \ref{lemL2wd} is proved in Section 6, its proof depends on Section 5 (which uses Lemma \ref{lemSameParity}). Lemma \ref{lemT1} is proved is Section 8, its proof depends on Sections 4, 5, and 7. Finally, Lemma \ref{lemExamples} is proved is Section 9, its proof depends on Sections 5 and 7.    

Denote by ${\rm Emb}^6_B(S^2\times S^1\sqcup S^3)$ the set of isotopy classes of Brunnian embeddings $S^2\times S^1\sqcup S^3\rightarrow S^6$.
For each $m$, denote by $1$ the point $(1,0,\ldots,0)\in S^m$.

\begin{def3}
Define a map 
$$\varkappa:{\rm Emb}^6_B(S^2\times S^1\sqcup S^3)\rightarrow {\mathbb Z} \text{\quad by the formula\quad} \varkappa([f]):=-{\rm lk}(f(S^2\times 1), f(S^3))\in{\mathbb Z}.$$
Here ${\rm lk}(f(S^2\times 1), f(S^3)):=F(D^3)\overset{\cdot}{\cap} f(S^3)$, where $F:D^3\rightarrow S^6$ is a general position embedding such that $F|_{S^2}=f|_{S^2\times 1}$. Clearly, $\varkappa$ is well defined.
\end{def3}

To shorten the notation, we shall write $\varkappa(f)$ instead of $\varkappa([f])$.

\begin{def6}
Denote by $e_k$ the vector $(0,\ldots,0,1,0,\ldots,0)\in {\mathbb R}^m$ where $1$ is at the $k$-th position. 
Orient ${\mathbb R}^m$ and $D^m$ so that the frame $e_1,\ldots, e_m$ is positive.
The {\it standard framing} of the standard embedding $S^{p-1}\subset S^{m-1}$ at any point is $(e_{p+1},e_{p+2},\ldots,e_m)$. The {\it standard framing} of $\std(D^3\times S^1)\subset S^6$ is the restriction of the standard framing of $S^4\supset \std(D^3\times S^1)$.
\end{def6}

\begin{def17}
By $M_{e_{i_1},e_{i_2},\ldots, e_{i_k}}$ we denote the framed submanifold $M\subset S^m$ with the framing $e_{i_1},e_{i_2},\ldots, e_{i_k}$. E.g., $\std(D^3\times S^1)_{e_6,e_7}$ denotes the standardly framed embedding $\std(D^3\times S^1)\subset S^6$.
\end{def17}

An embedding $f:S^2\times S^1\sqcup S^3\rightarrow S^6$ is called {\it simple} if 
$$f \text{\quad is Brunnian and\quad} f|_{S^2\times S^1}=\std|_{S^2\times S^1}.$$

\begin{def7}
Denote by $C_{\partial\std}$ the closure of the complement to a small tubular neighborhood of $\std(S^2\times S^1)$ in $S^6$. For a simple embedding $f:S^2\times S^1\sqcup S^3 \rightarrow S^6$, denote by 
$${\delet f}:S^3\rightarrow C_{\partial\std} \text{\quad the abbreviation\footnotemark{} of\quad} f.$$
\end{def7}

\footnotetext{If $A\subset X$ and $B\subset Y$, then for every $f:X\rightarrow Y$ such that $f(A)\subset B$ we have a map $g:A\rightarrow B, g(x)=f(x)$, which is called the {\it abbreviation} of $f$. Here we follow the definition of the abbreviation given in \cite[Chapter II, Submaps, p.58]{Vi}.}

\begin{def15}
For integers $a$ and $c$ such that $c|a$, let $\rho_{a,c}:{\mathbb Z}_a\rightarrow {\mathbb Z}_c$ be reduction modulo $c$ \footnote{Here ${\mathbb Z}_a$ denotes ${\mathbb Z}/a{\mathbb Z}$. Thus ${\mathbb Z}_0$ is ${\mathbb Z}$.}. Let $b$ be an integer such that $c|b$. Then for any $x\in {\mathbb Z}_a$ and $y\in {\mathbb Z}_b$
 $$x\equiv y \pmod{c} \text{\quad indicates that\quad} \rho_{a,c}(x)=\rho_{b,c}(y).$$
\end{def15}

\begin{def16}
Given a map $\phi:S^3\rightarrow S^2$, denote by $\hp(\phi)\in{\mathbb Z}$ its Hopf invariant. Given a framed $1$-submanifold $a\subset S^3$, denote by $\hp(a)\in{\mathbb Z}$ its Hopf invariant \footnote{Slightly shift $a$ along the first vector of its framing. Denote by $a'$ the obtained oriented $1$-submanifold of $S^3$. The Hopf invariant of $a$ is the linking number ${\rm lk}(a,a')$. Equivalently, the Hopf invariant of $a$ is the Hopf invariant of the map $S^3\rightarrow S^2$ obtained from $a$ by the Pontryagin-Thom construction.}.
\end{def16}

\begin{def8}
Let $f:S^2\times S^1\sqcup S^3\rightarrow S^6$ be a simple embedding such that $\std(D^3\times S^1)$ and ${\delet f}(S^3)$ are in general position. Let $\widehat{\nu}(f)\in{\mathbb Z}$ be the value of the Hopf invariant of the framed $1$-submanifold ${\delet f}^{-1}(\std(D^3\times S^1)_{e_6,e_7})$ of $S^3$. I.e., $$\widehat{\nu}(f) := \hp{\delet f}^{-1}(\std(D^3\times S^1)_{e_6,e_7}) \in{\mathbb Z}.$$
Let $$\nu([f]):=\rho_{0,2\varkappa(f)}(\widehat{\nu}(f))\in {\mathbb Z}_{2\varkappa(f)}.$$
\end{def8}

\begin{lemL2wd}
\label{lemL2wd}
The map $\nu:\varkappa^{-1}(k)\rightarrow {\mathbb Z}_{2k}$ is well defined for each $k$.
\end{lemL2wd}

Denote $$D^q_+:=\{(x_1,\ldots,x_{q+1})\in S^q|x_1\geq 0\} \text{\quad and \quad} D^q_-:=\{(x_1,\ldots,x_{q+1})\in S^q|x_1\leq 0\}.$$

An embedding $f:S^2\times S^1\sqcup S^3\rightarrow S^6$ is called {\it upper simple} if 
$$f \text{\quad is simple and\quad} f(S^3)\subset {\rm Int}D^6_+.$$

\begin{lemExistStandardized}
\label{lemExistStandardized}
Let $f$ be a simple embedding. Then there exists an isotopy fixed on $S^2\times S^1$ between $f$ and some upper simple embedding.
\end{lemExistStandardized}

\begin{proof}
Denote by $\Delta^2$ the $2$-disk $\{x_1\leq -\frac{1}{\sqrt{2}}, x_2=x_3=x_6=x_7=0\}$ in $S^6$. Then $\Delta^2\subset D^6_-$, ${\rm Int}\Delta^2\cap \std(S^2\times S^1)=\emptyset$ and $\partial\Delta^2=\std((-1, 0, 0)\times S^1)$. By general position, we may assume that $\Delta^2$ and $f(S^3)$ are disjoint.

Consider an embedding $U:D^6_-\rightarrow {\mathbb R}^6$ such that $U(D^6)$ is convex, $U(\Delta^2)=D^2$, restriction $U\std|_{D^2_-\times x}$ is linear for every $x\in S^1$, and $U(\std(D^2_-\times S^1))$ is orthogonal to $D^2$ at each point of $U(\partial\Delta^2)=U(\std((-1, 0, 0)\times S^1))=\partial D^2$ (Fig.~\ref{figupperSimple}).\footnote{Such $U$ exists only because the embedding $D^2_-\times S^1\xrightarrow{\std} D^6_-$ is standard.} To prove the Lemma it suffices to ``push'' $U(f(S^3))$ out of $U(D^6_-)$ via a smooth isotopy.

To do it we ``push'' each point of $U(f(S^3))$ out of $U(D^6_-)$ along the shortest path connecting the point and $D^2$ (Fig.~\ref{figupperSimple}, right). The path is non-zero because $U(f(S^3))$ is disjoint from $D^2$; the path is unique because $D^2$ is standard.
\end{proof}

	\begin{figure}[H]
	\centering
	\includegraphics[width=100mm]{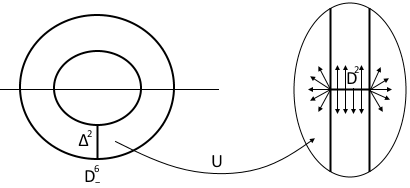}
	\caption{Embedding $U:D^6_-\rightarrow {\mathbb R}^6$ (proof of Lemma \ref{lemExistStandardized}).}
	\label{figupperSimple}
	\end{figure}
	
\begin{rem}
Lemma \ref{lemExistStandardized} is analogous to part (a) of Standartization Lemma 2.1 in \cite{Sk15} but does not directly follow from it, because the hypothesis of Lemma \ref{lemExistStandardized} does not satisfy the dimension restrictions of the hypothesis of Standartization Lemma.
\end{rem}

\begin{lemExistsDisk}
\label{lemExistsDisk}
Let $f:S^3\rightarrow S^6$ be an embedding isotopic to the standard embedding and such that $f(S^3)\subset {\rm Int}(D^6_+)$. Then there exists an embedded disk $\Delta\subset {\rm Int}(D^6_+)$ such that $\partial\Delta=f(S^3)$.
\end{lemExistsDisk}
\begin{proof}
Since $f$ is isotopic to the standard embedding, there exists an embedded disk $\Delta\subset S^6$ such that $\partial\Delta=f(S^3)$. Take a small ball $B^6\subset D^6_-$ disjoint from $\Delta$. Since $f(S^3)\subset {\rm Int}(D^6_+)$, it follows that there is an ambient isotopy of $S^6$ mapping $B^6$ to $D^6_-$ and fixed on $f(S^3)$. The image of $\Delta$ under the isotopy is as required.
\end{proof}

\begin{def10}
The embedding $\theta:{\rm Int}D^2_+\times S^1\rightarrow S^3$ is given by the formula:
$$\theta((x_1,x_2,x_3),(y_1,y_2)) := (x_2,x_3,y_1\sqrt{2-x_2^2-x_3^2},y_2\sqrt{2-x_2^2-x_3^2})/\sqrt{2}.$$
The image of $\theta$ is a tubular neighborhood $\{x_1^2+x^2_2<\frac{1}{2}\}\subset S^3$ of the circle $\{x_3^2+x_4^2=1\}\subset S^3$.
\end{def10}

\begin{def9}
Let $f$ be a upper simple embedding. Lemma \ref{lemExistsDisk} asserts that there is a $4$-disk $\Delta\subset {\rm Int}D^6_+$ such that $\partial\Delta={\delet f}(S^3)$.

Frame $\Delta$ arbitrarily so that the natural orientation of $\Delta$ agrees with the orientation of its boundary ${\delet f}(S^3)$. We may assume that $\Delta$ is in general position w.r.t. $\std(S^2\times S^1)$.

Note that the framed intersection $\Delta \cap \std(S^2\times S^1)$ lies in $\std({\rm Int}D^2_+\times S^1)$. Therefore, we have $\std^{-1}|_{S^2\times S^1}(\Delta)\subset {\rm Int}D^2_+\times S^1$ and $\theta\std^{-1}|_{S^2\times S^1}(\Delta)\subset S^3$. Define 
$$\widehat{\mu}(f) := \hp\theta\std^{-1}|_{S^2\times S^1}(\Delta)\in{\mathbb Z}.$$
\end{def9}

\begin{lemL1hatwd}
\label{lemL1hatwd}
The number $\widehat{\mu}(f)$ is well defined, i.e., does not depend on the choice of $\Delta$.
\end{lemL1hatwd}

\begin{def11}
Consider an isotopy class from ${\rm Emb}^6_B(S^2\times S^1\sqcup S^3)$. By Lemma \ref{lemExistStandardized}, this class contains a upper simple embedding $f$. Let $$\mu([f]):=\rho_{0,2\varkappa(f)}(\widehat{\mu}(f))\in {\mathbb Z}_{2\varkappa(f)}.$$
\end{def11}

\begin{lemL1wd}
\label{lemL1wd}
The map $\mu:\varkappa^{-1}(k)\rightarrow {\mathbb Z}_{2k}$ is well defined for each $k$.
\end{lemL1wd}

To shorten the notation, we shall write $\nu(f)$ and $\mu(f)$ instead of $\nu([f])$ and $\mu([f])$, respectively. It is important to remember, however, that $\nu(f)$ and $\mu(f)$ depend only on the isotopy class of $f$ while $\widehat{\nu}(f)$ and $\widehat{\mu}(f)$ depend also on the choice of the map $f$.

\begin{lemExamples}
\label{lemExamples}
For any integers $k$, $m$, $n$ such that $m\equiv n \pmod{2}$:
\begin{itemize}
\item{$\varkappa(f_{k,m,n})=k$}
\item{$\mu(f_{k,m,n}) \equiv m \pmod{2k}$}
\item{$\nu(f_{k,m,n}) \equiv n \pmod{2k}$}
\end{itemize}
\end{lemExamples}

\begin{lemSameParity}
\label{lemSameParity}
Let $f:S^2\times S^1\sqcup S^3\rightarrow S^6$ be a simple embedding. Then $\mu(f)\equiv\widehat{\nu}(f) \pmod{2}$.
\end{lemSameParity}

\begin{lemT1}
\label{lemT1}
Let $f,g:S^2\times S^1\sqcup S^3\rightarrow S^6$ be Brunnian embeddings. Suppose that $\varkappa(f)=\varkappa(g)$, $\mu(f)=\mu(g)$ and $\nu(f)=\nu(g)$. Then $f$ is isotopic to $g$.
\end{lemT1}

\begin{proof}[Proof of Theorem \ref{mainthm} modulo Lemmas \ref{lemExamples}, \ref{lemSameParity}, and \ref{lemT1}]
Let $f:S^2\times S^1\sqcup S^3\rightarrow S^6$ be a Brunnian embedding. Let $m$ and $n$ be integers such that $\rho_{0,2\varkappa(f)}(m)= \mu(f)$ and $\rho_{0,2\varkappa(f)}(n)= \nu(f)$. Lemma \ref{lemSameParity} asserts that $m \equiv n \pmod{2}$. By Lemma \ref{lemExamples}, embeddings $f$ and $f_{\varkappa(f),m,n}$ have the same $\varkappa$, $\mu$ and $\nu$ invariants. Hence $f$ is isotopic to $f_{\varkappa(f),m,n}$ by Lemma \ref{lemT1}.

By Lemma \ref{lemExamples}, the embeddings $f_{k,m,n}$, $f_{k,m+2k,n}$, and $f_{k,m,n+2k}$ have the same $\varkappa$, $\mu$ and $\nu$ invariants. Hence they are isotopic to each other by Lemma \ref{lemT1}.

Suppose that the embeddings $f_{k,m,n}$ and $f_{k',m',n'}$ are isotopic. Then they have the same $\varkappa$, $\mu$, and $\nu$ invariants. Hence by Lemma \ref{lemExamples}, it follows that $k=k'$, $m\equiv m' \pmod{2k}$ and $n\equiv n' \pmod{2k}$.
\end{proof}

\begin{def20}
Let $H:M\times I\rightarrow N\times I$ be an isotopy. For any $t\in I$ let ${\rm pr}_{M,t}:M\times t\rightarrow M$ and ${\rm pr}_{N,t}:N\times t\rightarrow N$ be the projections on the first factor. Denote by $H_t$ the map ${\rm pr}_{N,t}\circ H|_{M\times t}\circ {\rm pr}^{-1}_{M,t}:M\rightarrow N$.
\end{def20}

\subsection*{A relative framed cobordism and the relative Pontryagin-Thom isomorphism.}~\\
Let $M^m$ be a compact orientable manifold with boundary. Let $N_0,N_1\subset M$ be properly embedded framed submanifolds. A {\it relative framed cobordism} between $N_0$ and $N_1$ is a pair $(W, V)$ of framed submanifolds of $M\times I$ and $\partial M\times I$, respectively, such that
\begin{itemize}
\item{$V$ is a framed cobordism between $\partial N_0$ and $\partial N_1$.}
\item{$W$ is a proper submanifold with corners, i.e., $\partial W=N_0\cup_{\partial N_0}V\cup_{\partial N_1} N_1$, $W\cap \partial(M\times I)=\partial W$ and $W$ is transverse to $\partial(M\times I)$ at any point of $\partial W$.}
\item{The restriction of the framing of $W$ to $V$ is the framing of $V$.}
\item{The restriction of the framing of $W$ to $N_0$ and $N_1$ are the framings of $N_0$ and $N_1$, respectively.}
\end{itemize}
Submanifolds $N_0$ and $N_1$ are {\it relative framed cobordant} if a relative framed cobordism between them exists.

Let $N^n\subset M$ be a properly embedded framed submanifold. We may use the relative Pontryagin-Thom construction to produce a map $f:M\rightarrow S^{m-n}$ such that $N$ is the framed preimage of a regular point of $f$. We shall say that under the relative Pontryagin-Thom isomorphism, the framed submanifold $N$ {\it corresponds} to the map $f$.

Suppose that $N'\subset M$ is another properly embedded framed submanifold corresponding to a map $f':M\rightarrow S^{m-n}$. 
Then analogously to the absolute situation, $N$ is relative framed cobordant to $N'$ if and only if $f$ is homotopic to $f'$.

\medskip
\section{Proof of Lemmas \ref{lemL1hatwd} and \ref{lemL1wd} (that $\widehat{\mu}$ and $\mu$ are well-defined).}

\begin{lemDiskCobordism}
\label{lemDiskCobordism}
Let $f:S^3\rightarrow S^6$ be an embedding. Denote by $C_f$ the closure of the complement to a small tubular neighborhood of $f(S^3)$ in $S^6$. Let $\Delta\subset S^6$ and $\Delta'\subset S^6$ be $4$-disks such that $\partial \Delta=\partial \Delta'=f(S^3)$. Suppose that $\Delta$ and $\Delta'$ are framed. Suppose also that the natural orientations of $\Delta$ and $\Delta'$ agree with the orientations of their boundary $f(S^3)$.

Then the framed submanifolds $\Delta\cap C_f$ and $\Delta'\cap C_f$ are relative framed cobordant in $C_f$.
\end{lemDiskCobordism}

\begin{proof}
Under the relative Pontryagin-Thom isomorphism, the framed submanifolds $\Delta\cap C_f$ and $\Delta'\cap C_f$ of $C_f$ correspond to some maps $\delta:C_f\rightarrow S^2$ and $\delta':C_f\rightarrow S^2$, respectively. To prove the Lemma, it suffices to prove that $\delta$ is homotopic to $\delta'$.

Let $G:D^3\rightarrow S^6$ be an embedding such that the disk $G(D^3)$ intersects $f(S^3)$ transversally and the intersection $f(S^3)\cap G(D^3)$ is a single positive point. Let $g:S^2\rightarrow C_f$ be the abbreviation of $G|_{S^2}$. Clearly, ${\rm lk}(f(S^3), g(S^2))=1$.

By Alexander duality, $[g(S^2)]$ is a generator of $H_2(C_f;{\mathbb Z})\cong {\mathbb Z}$ and the remaining homology groups of $C_f$ are trivial (except $H_0$). The space $C_f$ is simply connected. Hence by Whitehead Theorem, $g:S^2\rightarrow C_f$ is a homotopy equivalence. Therefore, $g_*:[C_f, S^2]\rightarrow [S^2, S^2]$ is a bijection.

We have that ${\rm deg}(\delta\circ g)=\Delta \overset{\cdot}{\cap}g(S^2)$ and ${\rm deg}(\delta'\circ g)=\Delta' \overset{\cdot}{\cap}g(S^2)$. On the other hand, $\Delta \overset{\cdot}{\cap}g(S^2)=\Delta' \overset{\cdot}{\cap}g(S^2)=1$ because the orientations of $\Delta$ and $\Delta'$ agree with the orientations of their boundary $f(S^3)$. So ${\rm deg}(\delta\circ g)={\rm deg}(\delta'\circ g)$. Hence $\delta\circ g:S^2\rightarrow S^2$ is homotopic to $\delta'\circ g:S^2\rightarrow S^2$. Since $g_*:[C_f, S^2]\rightarrow [S^2, S^2]$ is a bijection, it follows that $\delta$ is homotopic to $\delta'$.
\end{proof}

\begin{proof}[Proof of Lemma \ref{lemL1hatwd}]
Let $\Delta\subset {\rm Int}D^6_+$ and $\Delta'\subset {\rm Int}D^6_+$ be $4$-disks such that $\partial \Delta=\partial \Delta'={\delet f}(S^3)$. Suppose that $\Delta$ and $\Delta'$ are framed. Suppose also that the natural orientations of $\Delta$ and $\Delta'$ agree with the orientations of their boundary ${\delet f}(S^3)$.

Denote by $C_{\delet f}$ the closure of the complement to a small tubular neighborhood of ${\delet f}(S^3)$ in $S^6$.

Lemma \ref{lemDiskCobordism} asserts that $\Delta\cap C_{\delet f}$ and $\Delta'\cap C_{\delet f}$ are relative framed cobordant in $C_{\delet f}$. The disk $D^6_-$ is a closure of a tubular neighborhood of a point. Hence, by general position, we may assume that $\Delta\cap C_{\delet f}$ and $\Delta'\cap C_{\delet f}$ are relative framed cobordant in $C_{\delet f}\cap {\rm Int}D^6_+$. The framed intersection of $\std({\rm Int}D^2_+\times S^1)$ and this relative framed cobordism is a framed cobordism between $K:=\Delta\cap\std({\rm Int}D^2_+\times S^1)$ and $K':=\Delta'\cap\std({\rm Int}D^2_+\times S^1)$. Therefore, $K$ and $K'$ are framed cobordant in $\std({\rm Int}D^2_+\times S^1)$. Hence $\hp\theta\std^{-1}(K)=\hp\theta\std^{-1}(K')$. The Lemma follows.
\end{proof}

The following Lemma \ref{lemEta} is a corollary of \cite[Theorem 1 (L.S. Pontryagin)]{Ce07}. We follow the notation of \cite{Ce07} in the statement and the proof of the Lemma.
\begin{lemEta}
\label{lemEta}
Let $L_1$ and $L_2$ be framed $1$-submanifolds of ${\rm Int}D^2_+\times S^1\subset S^2\times S^1$ homologous to $k[1\times S^1]$ for some integer $k$. Suppose that $L_1$ is framed cobordant to $L_2$ in $S^2\times S^1$.

Then $\hp(\theta(L_1)) \equiv \hp(\theta(L_2)) \pmod{2k}$.
\end{lemEta}

\begin{proof}
In codimension $2$ any integer homology between codimension $3$ submanifolds can be realized by a submanifold. Since $L_1$ and $L_2$ are homologous, it follows that there is a (not framed) cobordism $L$ in ${\rm Int}D^2_+\times S^1\times I$ between them. The definition of the {\it relative normal Euler class} $\overline{e}(L)\in {\mathbb Z}$ is given in \cite[\S 2]{Ce07}. We have $$\hp(\theta(L_2))-\hp(\theta(L_1))=\overline{e}(\theta(L))=\overline{e}(L)\equiv 0 \pmod{2k},$$
where the first two equations are obvious and the last follows by (2) in the proof of Theorem 1 in \cite{Ce07}.
\end{proof}

\begin{proof}[Proof of Lemma \ref{lemL1wd}]
Suppose that $f$ and $f'$ are isotopic upper simple embeddings.

Let $\Delta,\Delta'\subset {\rm Int}D^6_+$ be framed $4$-disks such that $\partial \Delta={\delet f}(S^3)$ and $\partial \Delta'={\delet f}'(S^3)$. Suppose also that $\Delta$ and $\Delta'$ are framed. Suppose also that the natural orientations of $\Delta$ and $\Delta'$ agree with the orientations of their boundaries ${\delet f}(S^3)$ and ${\delet f}'(S^3)$.

We shall prove that $\std^{-1}|_{S^2\times S^1}(\Delta)$ is framed cobordant to $\std^{-1}|_{S^2\times S^1}(\Delta')$ in $S^2\times S^1$. By Lemma \ref{lemEta} this would imply the assertion of the Lemma.

Let $H:S^6\times I\rightarrow S^6\times I$ be an isotopy between $f$ and $f'$. Notation $X\underset{fr}{=}Y$ means that $X$ is framed cobordant to $Y$. We have $$ \std^{-1}|_{S^2\times S^1}(\Delta) \overset{(1)}{\underset{fr}{=}} (H_1\circ\std|_{S^2\times S^1})^{-1}(H_1\Delta)  \overset{(2)}{\underset{fr}{=}}  \std^{-1}|_{S^2\times S^1}(H_1\Delta) \overset{(3)}{\underset{fr}{=}} \std^{-1}|_{S^2\times S^1}(\Delta').$$

The equations $(1)$, $(2)$, and $(3)$ hold for the following reasons. 
\begin{itemize}
\item[(1)]{Here the framed cobordism is the $\std|_{S^2\times S^1}$-preimage of $H\Delta$.}
\item[(2)]{Both $f$ and $f'$ are simple, so the map $H_1:S^6\rightarrow S^6$ is identical on $\std(S^2\times S^1)$.}
\item[(3)]{Denote by $C_{{\delet f}'}$ the closure of the complement to a small tubular neighborhood of ${\delet f}'(S^3)$ in $S^6$. The equation holds since $(H_1\Delta)\cap C_{{\delet f}'}$ and $\Delta'\cap C_{{\delet f}'}$ are relative framed cobordant in $C_{{\delet f}'}$ by Lemma \ref{lemDiskCobordism}}.
\end{itemize}
\end{proof}

\medskip
\section{Additional lemmas, proof of Lemma \ref{lemSameParity} (that $\mu$ and $\nu$ have the same parity).}
By $S^3_1$ and $S^3_2$ we denote two distinct copies of $S^3$.
\begin{def12}
Consider a Brunnian embedding $g:S^3_1\sqcup S^3_2\rightarrow S^6$. Let $\Delta$ be a general position embedded $4$-disk such that $\partial \Delta=g(S^3_2)$. Frame $\Delta$ arbitrarily so that its natural orientation agrees with the orientation of its boundary $g(S^3_2)$.

Define 
$$\mu([g]) := \hp({g|_{S^3_1}}^{-1}(\Delta)).$$
The definition of $\nu([g])$ is analogous and is obtained by exchanging the components.
\end{def12}

To shorten the notation, we shall write $\nu(g)$ and $\mu(g)$ instead of $\nu([g])$ and $\mu([g])$, respectively.

Our definition of $\mu$ and $\nu$ is equivalent to the definition in \cite{Atl2} of $\lambda_{12}:{\rm Emb}^6(S^3\sqcup S^3)\rightarrow {\mathbb Z}$ and $\lambda_{21}:{\rm Emb}^6(S^3\sqcup S^3)\rightarrow {\mathbb Z}$, respectively. 

Recall that $\lambda_{12}(\omega)=0$, $\lambda_{21}(\omega)=2$, $\lambda_{12}(\tau(\eta))=1$ (see \cite{Atl2}) and $\lambda_{21}(\tau(\eta))=1$ (see \cite[proof of the $\tau$-relation, end of \S 4]{Sk07}). Hence, $\mu(\omega)=0$, $\nu(\omega)=2$, $\mu(\tau(\eta))=1$ and $\nu(\tau(\eta))=1$ 

\begin{lemLambda}
\label{lemLambda}
Let $m$ and $n$ be numbers such that $m\equiv n\pmod{2}$. Then $\mu(g_{m,n})=m$ and $\nu(g_{m,n})=n$.
\end{lemLambda}
\begin{proof}
Obviously, the invariants $\mu$ and $\nu$ are additive with respect to $\#$. Hence, $\mu(g_{m,n})=\mu(m\tau(\eta)\#\frac{(n-m)}{2}\omega)=m$ and $\nu(g_{m,n})=\nu(m\tau(\eta)\#\frac{(n-m)}{2}\omega)=n$.
\end{proof}

Let $\sigma:S^3\rightarrow S^6$ be an embedding whose image is $$ \{x_1\geq 0, x_1^2+x_2^2+x_3^2=\frac{1}{2}, x_4^2+x_5^2=\frac{1}{2}\}\cup \{x_1<0, x_1^2+x_4^2+x_5^2=\frac{1}{2}, x_2^2+x_3^2=\frac{1}{2}\}.$$ Informally, $\sigma(S^3)$ is the result of an embedded $1$-surgery on the circle $\std((-1)\times S^1)$ of $\std(S^2\times S^1)$.

We shall use the fact that in $S^4\subset S^6$ the sphere $\sigma(S^3)\subset S^4$ bounds the $4$-disk
$$\Delta_\sigma:= \{x_1\geq 0, x_1^2+x_2^2+x_3^2\leq\frac{1}{2}, x_4^2+x_5^2\geq\frac{1}{2}\}\cup \{x_1<0, x_1^2+x_4^2+x_5^2\geq\frac{1}{2}, x_2^2+x_3^2\leq\frac{1}{2}\}.$$

Hence the embedding $\sigma$ is isotopic to the standard embedding. We have that $\sigma(S^3)\cap D^6_+=\std(D^2_+\times S^1)$ and $\Delta_\sigma\cap D^6_+=\std(D^3\times S^1)\cap D^6_+\subset S^4$. Orient $\Delta_\sigma$ so that its orientation coincides with the orientation of $S^4$ at any point of $S^4\cap \Delta_\sigma$. We may assume (by changing $\sigma$ if necessary) that the orientation of $\Delta_\sigma$ agrees with the orientation of its boundary $\sigma(S^3)$.

\begin{lemSameLNumbers}
\label{lemSameLNumbers}
Let $f$ be a upper simple embedding. Denote by $g$ the embedding $\sigma\sqcup f|_{S^3}:S^3_1\sqcup S^3_2\rightarrow S^6$. Then $g$ is Brunnian and $\widehat{\mu}(f)=\mu(g)$ and $\widehat{\nu}(f)=\nu(g)$.
\end{lemSameLNumbers}

\begin{proof}
By Lemma \ref{lemExistsDisk}, there exists a general position embedded $4$-disk $\Delta\subset {\rm Int}D^6_+$ such that $\partial \Delta=f(S^3)=g(S^3_2)$. Frame $\Delta$ arbitrarily so that its natural orientation agrees with the orientation of its boundary $f(S^3)=g(S^3_2)$. We have 
$$\widehat{\mu}(f) \overset{(1)}{=} \hp\theta\std^{-1}|_{S^2\times S^1}(\Delta) \overset{(2)}{=} \hp\sigma^{-1}\std|_{{\rm Int}D^2_+\times S^1}\std^{-1}|_{S^2\times S^1}(\Delta) \overset{(3)}{=} \hp\sigma^{-1}(\Delta) \overset{(4)}{=} \mu(g), \text{ where}$$

\begin{itemize}
\item[(1)]{is the definition of $\widehat{\mu}(f)$.}
\item[(2)]{holds, because $\theta|_{{\rm Int}D^2_+\times S^1}$ is isotopic to $\sigma^{-1}\std|_{{\rm Int}D^2_+\times S^1}:{\rm Int}D^2_+\times S^1\rightarrow S^3$.}
\item[(3)]{holds, because $\sigma^{-1}\std|_{{\rm Int}D^2_+\times S^1}\std^{-1}|_{S^2\times S^1}(\Delta)=\sigma^{-1}(\Delta)$.}
\item[(4)]{is the definition of $\mu(g)$.}
\end{itemize}

Consider the framed submanifold $\Delta_{\sigma, e_6, e_7}$. The framings of $\Delta_{\sigma, e_6, e_7}$ and of $\std(D^3\times S^1)_{e_6, e_7}$ coincide on $\Delta_\sigma\cap D^6_+=\std(D^3\times S^1)\cap D^6_+$. We have

$$ \widehat{\nu}(f) \overset{(1)}{=} \hp{\delet f}^{-1}(\std(D^3\times S^1)_{e_6, e_7}) \overset{(2)}{=} \hp{\delet f}^{-1}(\Delta_{\sigma, e_6, e_7}) \overset{(3)}{=} \hp g|_{S^3_2}^{-1}(\Delta_{\sigma, e_6, e_7}) \overset{(4)}{=} \nu(g), \text{ where}$$

\begin{itemize}
\item[(1)]{is the definition of $\widehat{\nu}(f)$.}
\item[(2)]{holds, because the framings of $\Delta_{\sigma, e_6, e_7}$ and of $\std(D^3\times S^1)_{e_6, e_7}$ coincide on $\Delta_\sigma\cap D^6_+=\std(D^3\times S^1)\cap D^6_+$.}
\item[(3)]{holds, because ${\delet f}=g|_{S^3_2}$ by definition of $g$.}
\item[(4)]{is the definition of $\nu(g)$.}
\end{itemize}
\end{proof}

\begin{lemFixedIsotopy}
\label{lemFixedIsotopy}
Let $f$ and $f'$ be simple embeddings. Suppose that there is an isotopy between $f$ and $f'$ fixed on $S^2\times S^1$. Then $\widehat{\nu}(f)=\widehat{\nu}(f')$.
\end{lemFixedIsotopy}
\begin{proof}
Since every smooth isotopy is ambient, there is an ambient isotopy $H:S^6\times I \rightarrow S^6\times I$ between $f$ and $f'$ fixed on $S^2\times S^1$.

By definition, $\widehat{\nu}(f)$ and $\widehat{\nu}(f')$ are the values of the Hopf invariants of the framed $1$-submanifolds ${\delet f}^{-1}(\std(D^3\times S^1)_{e_6, e_7})$ and ${\delet f}^{\prime -1}(\std(D^3\times S^1)_{e_6, e_7})$ of $S^3$. 

Therefore it suffices to prove that ${\delet f}^{-1}(\std(D^3\times S^1)_{e_6, e_7})$ and ${\delet f}^{\prime -1}(\std(D^3\times S^1)_{e_6, e_7})$ are framed cobordant. The required framed cobordism is 
$$({\delet f}^{-1}\circ H^{-1})((\std(D^3\times S^1)\times I)_{e_6, e_7} \cap H({\delet f}(S^3)\times I)).$$
\end{proof}

\begin{proof}[Proof of Lemma \ref{lemSameParity}]
Lemma \ref{lemExistStandardized} asserts that there exists an isotopy fixed on $S^2\times S^1$ between $f$ and some upper simple embedding $f'$. Denote by $g$ the embedding $\sigma\sqcup f'|_{S^3}:S^3\sqcup S^3\rightarrow S^6$. Then $g$ is Brunnian and
$$\mu(f) \overset{(1)}{\equiv} \mu(f') \overset{(2)}{\equiv} \widehat{\mu}(f') \overset{(3)}{\equiv} \mu(g) \overset{(4)}{\equiv} \nu(g) \overset{(5)}{\equiv} \widehat{\nu}(f') \overset{(6)}{\equiv} \widehat{\nu}(f) \pmod{2}, \text{ where}$$

\begin{itemize}
\item[(1)]{holds, because $\mu$ is well-defined and $f$ is isotopic to $f'$.}
\item[(2)]{holds by the definition of $\mu$.}
\item[(3,5)]{hold by Lemma \ref{lemSameLNumbers}.}
\item[(4)]{by Theorem \ref{Hthm}, $g$ is isotopic to $g_{m,n}$ for some integers $m\equiv n \pmod{2}$. Lemma \ref{lemLambda} asserts that $\mu(g)=m$ and $\nu(g)=n$. Hence the equality holds.\footnote{Here it is convenient for us to deduce the fact that $\mu(g) \equiv \nu(g) \pmod{2}$ from Theorem \ref{Hthm}. However, the proof of this fact is an essential part of the proof of Theorem \ref{Hthm}.}}
\item[(6)]{holds by Lemma \ref{lemFixedIsotopy}, since the isotopy between $f$ and $f'$ is fixed on $S^2\times S^1$.}
\end{itemize}
\end{proof}

\medskip
\section{On the homotopy type of $C_{\partial\std}$.}
Under the relative Pontryagin-Thom isomorphism, the framed submanifolds $\std(D^3\times S^1)_{e_6, e_7}\cap C_{\partial\std}$ and $\std(D^3\times 1)_{e_5, e_6, e_7}\cap C_{\partial\std}$ of $C_{\partial\std}$ correspond to some maps 
$$p_2:C_{\partial\std}\rightarrow S^2 \text{\quad and\quad} p_3:C_{\partial\std}\rightarrow S^3,$$ respectively. 

Define an embedding $z'_2:S^2\rightarrow C_{\partial\std}$ by the formula
$$z_2'(x_1,x_2,x_3)=(0,0,0,0,x_1,x_2,x_3).$$

Define a map $z_2:S^3\rightarrow C_{\partial\std}$ to be the composition
$$S^3\xrightarrow{\text{Hopf map } \eta} S^2 \xrightarrow{z'_2} C_{\partial\std}.$$

Define an embedding $z_3:S^3\rightarrow C_{\partial\std}$ by the formula
$$z_3(x_1,x_2,x_3,x_4)=(0,0,0,x_1,x_2,x_3,x_4).$$

Given a map $\phi:S^3\rightarrow C_{\partial\std}$, we denote by $[\phi]\in\pi_3(C_{\partial\std})$ its homotopy class\footnote{The space $C_{\partial\std}$ is simply connected so we can ignore the base points.}.

\begin{lemBasis}
\label{lemBasis}
Homotopy classes $[z_2]$ and $[z_3]$ form a basis of $\pi_3(C_{\partial\std})$. For each map $\phi:S^3\rightarrow C_{\partial\std}$ the coefficients of its homotopy class $[\phi]$ in the basis are $\hp(p_2\phi)$ and ${\rm deg}(p_3\phi)$, i.e., $[\phi] = \hp(p_2\phi)[z_2]+{\rm deg}(p_3\phi)[z_3]$.
\end{lemBasis}
\begin{proof}
Define an embedding $z_4:S^4\rightarrow S^6$ by the formula
$$z_4(x_1,x_2,x_3,x_4,x_5):=(x_1,x_2,x_3,0,0,x_4,x_5).$$

Consider the following diagram.

\begin{tikzcd}
S^2
\arrow{d}{i_2}
& {} \\
S^2\vee S^3\vee S^4 \quad
\arrow{r}{z'_2\vee z_3\vee z_4}
&
\quad C_{\partial\std} 
\arrow[bend right]{ul}{p_2}
\arrow[swap, bend left]{dl}{p_3}
\\
S^3
\arrow[swap]{u}{i_3}
& {}
\end{tikzcd}

The space $S^2\vee S^3\vee S^4$ is obtained by gluing together the points $(0,0,1)\in S^2$, $(0,0,0,1)\in S^3$, and $(0,0,0,0,1)\in S^4$. The maps $i_2$ and $i_3$ are inclusions.

Let us now prove that the map $z'_2\vee z_3\vee z_4$ is a homotopy equivalence. The space $C_{\partial\std}$ is simply connected by general position. Note that $|{\rm lk}(z'_2(S^2), \std(S^2\times S^1))|=1$, $|{\rm lk}(z_3(S^3), \std(S^2\times 1))|=1$, and $|{\rm lk}(z_4(S^4), \std(1\times S^1))|=1$. So by Alexander duality, $[z'_2(S^2)]$, $[z_3(S^3)]$, and $[z_4(S^4)]$ are generators of $H_2(C_{\partial\std};{\mathbb Z})\cong {\mathbb Z}$, $H_3(C_{\partial\std};{\mathbb Z})\cong {\mathbb Z}$, and $H_4(C_{\partial\std};{\mathbb Z})\cong {\mathbb Z}$, respectively. Hence by Whitehead Theorem the map $z'_2\vee z_3\vee z_4$ is a homotopy equivalence.

By the Hilton Theorem on the homotopy groups of wedge product of spheres, $\pi_3(S^2\vee S^3\vee S^4)$ is isomorphic to ${\mathbb Z}\oplus {\mathbb Z}$ and its generators are $[i_2\eta]$ and $[i_3]$. Since $z'_2\vee z_3\vee z_4$ is a homotopy equivalence, it follows that $[(z'_2\vee z_3\vee z_4)i_2\eta]=[z_2]$ and $[(z'_2\vee z_3\vee z_4)i_3]=[z_3]$ form a basis of $\pi_3(C_{\partial\std})$.

Let $\phi_1,\phi_2:S^3\rightarrow C_{\partial\std}$ be maps such that $[\phi]=[\phi_1]+[\phi_2]$. Then $\hp(p_2\phi)=\hp(p_2\phi_1)+\hp(p_2\phi_2)$ and ${\rm deg}(p_3\phi)={\rm deg}(p_3\phi_1)+{\rm deg}(p_3\phi_2)$. Therefore it suffices to only prove the second assertion of the Lemma for the cases $\phi=z_2$ and $\phi=z_3$.

{\it Case $\phi=z_2$, proof of $\hp(p_2z_2)=1$.} Consider the map $p_2z'_2:S^2\rightarrow S^2$. By the definition of $p_2$, the map $p_2z'_2$ has a regular point whose framed preimage is $(z'_2)^{-1}(\std(D^3\times S^1)_{e_6, e_7}\cap C_{\partial\std})$. We have that $\std(D^3\times S^1)_{e_6, e_7}\cap {\rm Im}(z'_2)=(0,0,0,0,1,0,0)_{e_6, e_7}$, and $(z'_2)^{-1}((0,0,0,0,1,0,0)_{e_6, e_7})$ is a single positively framed point $(1,0,0)_{e_2,e_3}\in S^2$. Therefore ${\rm deg}(p_2z'_2)=1$. So $\hp(p_2z_2)=\hp(p_2z'_2\eta)={\rm deg}^2(p_2z'_2)=1$.

{\it Case $\phi=z_2$, proof of ${\rm deg}(p_3z_2)=0$.} Consider the map $p_3z_2:S^3\rightarrow S^3$. By the definition of $p_3$, we have that $\std(D^3\times 1)_{e_5, e_6, e_7}\cap C_{\partial\std}$ is the framed preimage of a regular point of $p_3$. However $\std(D^3\times 1)$ is disjoint from ${\rm Im}(z_2)={\rm Im}(z'_2)$ because $\std(D^3\times 1)\subset\{x_5=x_6=x_7=0\}$ and ${\rm Im}(z'_2)=\{x^2_5+x^2_6+x^2_7=1\}$. So the preimage of the corresponding regular point of $p_3z_2$ is empty. Therefore ${\rm deg}(p_3z_2)=0$.

{\it Case $\phi=z_3$, proof of $\hp(p_2z_3)=0$.} Consider the map $p_2z_3:S^3\rightarrow S^2$. By the definition of $p_2$, the map $p_2z_3$ has a regular point whose framed preimage is $z_3^{-1}(\std(D^3\times S^1)_{e_6, e_7}\cap C_{\partial\std})$. We have that $\std(D^3\times S^1)_{e_6, e_7}\cap {\rm Im}(z_3)=\{x_4^2+x_5^2=1\}_{e_6, e_7}$, and $z_3^{-1}(\{x_4^2+x_5^2=1\}_{e_6, e_7})$ is the framed circle $\{x_1^2+x_2^2=1\}_{e_3,e_4}\subset S^3$. The Hopf invariant of this framed circle is zero. Therefore $\hp(p_2z_3)=0$.

{\it Case $\phi=z_3$, proof of ${\rm deg}(p_3z_3)=1$.} Consider the map $p_3z_3:S^3\rightarrow S^2$. By the definition of $p_3$, the map $p_3z_3$ has a regular point whose framed preimage is $z_3^{-1}(\std(D^3\times 1)_{e_5, e_6, e_7}\cap C_{\partial\std})$. We have that $\std(D^3\times 1)_{e_5, e_6, e_7}\cap {\rm Im}(z_3)=(0,0,0,1,0,0,0)_{e_5, e_6, e_7}$, and $z_3^{-1}((0,0,0,1,0,0,0)_{e_5, e_6, e_7})$ is a single positively framed point $(1,0,0,0)_{e_2,e_3,e_4}\in S^3$. Therefore ${\rm deg}(p_3z_3)=1$.






\end{proof}

\begin{lemInBasis}
\label{lemInBasis}
Let $f$ be a simple embedding. Then $[{\delet f}] = \widehat{\nu}(f)[z_2] + \varkappa(f)[z_3]$.
\end{lemInBasis}
\begin{proof}
Lemma \ref{lemBasis} asserts that $[{\delet f}] = \hp(p_2{\delet f})[z_2]+{\rm deg}(p_3{\delet f})[z_3]$. It remains to prove that $\hp(p_2{\delet f})=\widehat{\nu}(f)$ and ${\rm deg}(p_3{\delet f})=\varkappa(f)$.

We have,
$$\hp(p_2{\delet f}) = \hp({\delet f}^{-1}(\std(D^3\times S^1)_{e_6,e_7})) = \widehat{\nu}(f).$$
Here the first equation holds since ${\delet f}^{-1}(\std(D^3\times S^1)_{e_6,e_7})$ is the framed preimage of a regular point of $p_2{\delet f}:S^3\rightarrow S^2$. The second equation is the definition of $\widehat{\nu}$.

We also have,
$${\rm deg}(p_3{\delet f}) \overset{(1)}{=} \std(D^3\times 1)_{e_5,e_6,e_7} \overset{\cdot}{\cap} {\delet f}(S^3) = {\rm lk}(\std(S^2\times 1), {\delet f}(S^3)) \overset{(2)}{=}  \varkappa(f), \text{ where}$$

\begin{itemize}
\item[(1)]{holds since ${\delet f}^{-1}(\std(D^3\times 1)_{e_5,e_6,e_7})$ is the framed preimage of a regular point of $p_3{\delet f}:S^3\rightarrow S^3$.}
\item[(2)]{is the definition of $\varkappa$.}
\end{itemize}

\end{proof}

\begin{lemTorIsotopy}
\label{lemTorIsotopy}
Let $H:S^6\times I \rightarrow S^6\times I$ be an isotopy such that $H_0|_{S^2\times S^1}=H_1|_{S^2\times S^1}=\std|_{S^2\times S^1}$. Suppose that $H_1(C_{\partial\std})=C_{\partial\std}$. Let $\widetilde{H}:C_{\partial\std}\rightarrow C_{\partial\std}$ be the abbreviation of $H_1$. Then

\begin{itemize}
\item[(I)]{$[\widetilde{H} z_2] = [z_2]$.}
\item[(II)]{$[\widetilde{H} z_3] = 2k[z_2]+[z_3]$ for some integer $k$.}
\end{itemize}
 
\end{lemTorIsotopy}

\begin{proof}
(I).
The map $\pi_2(C_{\partial\std})\rightarrow {\mathbb Z}$ given by the formula $[\phi]\mapsto {\rm lk}(\phi(S^2), \std(S^2\times S^1))$ is an isomorphism. The linking number is an isotopy invariant. Therefore $\widetilde{H}$ induces the identity isomorphism of $\pi_2(C_{\partial\std})$.

So, the maps $z'_2:S^2\rightarrow C_{\partial\std}$ and $\widetilde{H} z'_2:S^2\rightarrow C_{\partial\std}$ are homotopic. Hence, by the definition of $z_2$, the maps $z_2:S^3\rightarrow C_{\partial\std}$ and $\widetilde{H} z_2:S^3\rightarrow C_{\partial\std}$ are homotopic as well, i.e., $[\widetilde{H} z_2] = [z_2]$.

(II). Consider the simple embedding $f:S^2\times S^1\sqcup S^3\rightarrow S^6$ such that ${\delet f}=z_3$. We have, $$[\widetilde{H} z_3]=[\widetilde{H} {\delet f}] \overset{(1)}{=} \widehat{\nu}(\widetilde{H} f)[z_2] + \varkappa(\widetilde{H} f)[z_3] \overset{(2)}{=} \widehat{\nu}(\widetilde{H} f)[z_2] + [z_3], \text{ where}$$

\begin{itemize}
\item[(1)]{holds by Lemma \ref{lemInBasis}.}
\item[(2)]{the linking number is an isotopy invariant, so $\varkappa(\widetilde{H} f)=\varkappa(f)=1$.}
\end{itemize}

The integer $\widehat{\nu}(\widetilde{H} f)$ is even since
$$\widehat{\nu}(\widetilde{H} f) \overset{(1)}{\equiv} \mu(\widetilde{H} f) \overset{(2)}{\equiv} \mu(f) \overset{(3)}{\equiv} \widehat{\nu}(f) \overset{(4)}{\equiv} 0 \pmod{2}, \text{ where}$$
\begin{itemize}
\item[(1,3)]{follow from Lemma \ref{lemSameParity}.}
\item[(2)]{holds, because $\mu$ is well defined (Lemma \ref{lemL1hatwd}) and $\widetilde{H} f$ is isotopic to $f$.}
\item[(4)]{follows from Lemmas \ref{lemBasis} and \ref{lemInBasis} since ${\delet f}=z_3$.}
\end{itemize}

\end{proof}

\medskip
\section{Proof of Lemma \ref{lemL2wd} (that $\widehat{\nu}$ is well-defined).}
Let $f$ and $f'$ be isotopic simple embeddings. We shall prove that $\widehat{\nu}(f')-\widehat{\nu}(f)$ is divisible by $2\varkappa(f)=2\varkappa(f')$. This will imply that $\nu(f)$ is well defined.

Let $H:S^6\times I \rightarrow S^6\times I$ be an isotopy between $f$ and $f'$. Since the embeddings $f$ and $f'$ are simple, it follows that $H_0|_{S^2\times S^1}=H_1|_{S^2\times S^1}=\std|_{S^2\times S^1}$. Therefore we may assume that $H_1(C_{\partial\std})=C_{\partial\std}$. Let $\widetilde{H}:C_{\partial\std}\rightarrow C_{\partial\std}$ be the abbreviation of $H_1$. Let $\widetilde{H}^*:\pi_3(C_{\partial\std})\rightarrow \pi_3(C_{\partial\std})$ be the homomorphism induced by $\widetilde{H}$. Then for some integer $k$
\begin{multline*}
\widehat{\nu}(f')[z_2] + \varkappa(f)[z_3] \overset{(1)}{=} [{\delet f}'] = [\widetilde{H}{\delet f}] = \widetilde{H}^*[{\delet f}] \overset{(2)}{=} \widetilde{H}^*\big(\widehat{\nu}(f)[z_2] + \varkappa(f)[z_3]\big) = \\ \widehat{\nu}(f)[\widetilde{H} z_2] + \varkappa(f)[\widetilde{H} z_3] \overset{(3)}{=} \big(\widehat{\nu}(f)+2\varkappa(f)k\big)[z_2] + \varkappa(f)[z_3], \text{ where}
\end{multline*}

\begin{itemize}
\item[(1)]{holds by Lemma \ref{lemInBasis}.}
\item[(2)]{holds, because $[{\delet f}] = \widehat{\nu}(f)[z_2] + \varkappa(f)[z_3]$ by Lemma \ref{lemInBasis}.}
\item[(3)]{holds, because $[\widetilde{H} z_2] = [z_2]$ and $[\widetilde{H} z_3] = 2k[z_2]+[z_3]$ by Lemma \ref{lemTorIsotopy}.}
\end{itemize}

The homotopy classes $[z_2]$ and $[z_3]$ form a basis of $\pi_3(C_{\partial\std})$ by Lemma \ref{lemBasis}. So, $\widehat{\nu}(f')-\widehat{\nu}(f)=2\varkappa(f)k$. \qed

\medskip
\section{Framing an embedded connected sum.}

A {\it framed embedding} is an embedding whose image is framed.

\begin{lemFramedSum}
\label{lemFramedSum} 
Let $f:M\rightarrow S^q$ and $g:N\rightarrow S^q$ be framed embeddings of compact orientable connected manifolds, ${\rm dim}M={\rm dim}N$. Then the framing of $f(M_0)\sqcup g(N_0)$ extends to a framing of $(f\#g)(M\#N)$.
\end{lemFramedSum}

\begin{proof}
Denote $p:={\rm dim}M={\rm dim}N$.
Suppose that the connected sum $f\#g$ was constructed along a tube $l:D^p\times I\rightarrow S^q$. We may assume that the first vectors of the framings of $f(M)$ and $g(N)$ are tangent to ${\rm Im}(l)$ at any point of $l|_{D^p\times 0}$ and $l|_{D^p\times 1}$, respectively (i.e., that they ``look'' into ${\rm Im}(l)$).

The last $q-p-1$ vectors of the framing of $f(M_0)\sqcup g(N_0)$ can be extended from the ends $l(D^p\times 0)$ and $l(D^p\times 1)$ of the tube to a $(q-p-1)$-framing of the whole tube $l(D^p\times I)$.

The first vectors of the framing of $f(M_0)\sqcup g(N_0)$ can now be extended to $(f\#g)(M\#N)$ using the orientation of $(f\#g)(M\#N)$.
\end{proof}

\begin{def19}
Let $f,g:M^m\rightarrow S^k$ be framed embeddings. An isotopy $H:S^k\times I\rightarrow S^k\times I$ between $f$ and $g$ is a {\it framed isotopy} if for each point $x\in M$ and for each integer $1\leq i \leq k-m$ the $dH_1$-image of the $i$-th vector of the framing of $f(M)$ at $f(x)$ is the $i$-th vector of the framing of $g(M)$ at $g(x)$.

The isotopy $H$ is also called a {\it framed isotopy between the framed submanifolds} $f(M)$ and $g(M)$.
\end{def19}

\begin{lemFramedHandlebodyIsotopy}
\label{lemFramedHandlebodyIsotopy}
Let $f,g:D^n\times S^m\rightarrow S^k$ be framed embeddings. Let $B^n\subset D^n$ be an $n$-ball. Suppose that $f|_{B^n\times S^m}=g|_{B^n\times S^m}$ and that the framings of $f(B^n\times S^m)$ and of $g(B^n\times S^m)$ are the same.

Then $f$ is framed isotopic to $g$.
\end{lemFramedHandlebodyIsotopy}

The Lemma is obvious and is proved by the ``shrinking'' argument.






\medskip
\section{Proof of Lemma \ref{lemT1}.}
The following three lemmas are proved later in this section.

\begin{lemIsotopic}
\label{lemIsotopic}
Let $f$ and $f'$ be upper simple embeddings. Suppose that $\varkappa(f)=\varkappa(f')$, $\widehat{\mu}(f)=\widehat{\mu}(f')$ and $\widehat{\nu}(f)=\widehat{\nu}(f')$. Then $f$ is isotopic to $f'$.
\end{lemIsotopic}

\begin{lemChangeL1}
\label{lemChangeL1}
Let $f$ be a upper simple embedding. Then for any integer $c$ there exists a upper simple embedding $f'$ isotopic to $f$ and such that $\widehat{\mu}(f')=\widehat{\mu}(f)+2\varkappa(f)c$ and $\widehat{\nu}(f')=\widehat{\nu}(f)$.
\end{lemChangeL1}

\begin{lemChangeL2}
\label{lemChangeL2}
Let $f$ be a simple embedding. Then for any integer $c$ there exists a simple embedding $f'$ isotopic to $f$ and such that $\widehat{\nu}(f')=\widehat{\nu}(f)+2\varkappa(f)c$.
\end{lemChangeL2}

\begin{proof}[Proof of Lemma \ref{lemT1} modulo Lemmas \ref{lemIsotopic}, \ref{lemChangeL1}, and \ref{lemChangeL2}]
Without loss of generality, we may assume that both $f$ and $g$ are simple.

By Lemma \ref{lemChangeL2}, there are simple embeddings $f'$ and $g'$ isotopic to $f$ and $g$, respectively, and such that $\widehat{\nu}(f')=\widehat{\nu}(g')$.

By Lemma \ref{lemExistStandardized}, there is an isotopy fixed on $S^2\times S^1$ between $f'$ and some upper simple embedding $f''$. Likewise, there is an isotopy fixed on $S^2\times S^1$ between $g'$ and some upper simple embedding $g''$. Since the isotopy between $f'$ and $f''$ is fixed on $S^2\times S^1$, it follows by Lemma \ref{lemFixedIsotopy} that $\widehat{\nu}(f')=\widehat{\nu}(f'')$. Likewise, $\widehat{\nu}(g')=\widehat{\nu}(g'')$. Therefore, $\widehat{\nu}(f'')=\widehat{\nu}(g'')$.

By Lemma \ref{lemChangeL1}, there are upper simple embeddings $f'''$ and $g'''$ isotopic to $f''$ and $g''$, respectively, and such that $\widehat{\nu}(f''')=\widehat{\nu}(g''')$ and $\widehat{\mu}(f''')=\widehat{\mu}(g''')$. Clearly, $\varkappa(f''')=\varkappa(f)=\varkappa(g)=\varkappa(g''')$.

By Lemma \ref{lemIsotopic}, the embeddings $f'''$ and $g'''$ are isotopic. Hence $f$ and $g$ are isotopic as well.
\end{proof}

To prove Lemma \ref{lemIsotopic}, we need the following Lemma \ref{lemHthmFixed}. See its proof in \cite[proof of Theorem 7.1]{Ha66} or in \cite[proof of Claim 3.1]{MSk13}. More details should appear in the newer version of \cite{MSk13}.
\begin{lemHthmFixed}
\label{lemHthmFixed}
Let $g,g':S^3_1\sqcup S^3_2\rightarrow S^6$ be isotopic embeddings. Suppose that $g|_{S^3_1}=g'|_{S^3_1}$. Then there exists an isotopy between $g$ and $g'$ fixed on $S^3_1$.
\end{lemHthmFixed}

\begin{proof}[Proof of Lemma \ref{lemIsotopic}]
Denote by $C_{\sigma}$ the closure of the complement to a small tubular neighborhood of $\sigma(S^3)$ in $S^6$. Since $f,f'$ are upper simple, it follows that $f(S^3),f'(S^3)\subset C_{\sigma}$. Let $g,g':S^3\rightarrow C_{\sigma}$ be the abbreviations of $f|_{S^3}$ and $f'|_{S^3}$, respectively.

We shall deduce the Lemma from the following claims which are proved immediately after the Lemma.

\begin{itemize}
\item[Claim 1.]{There is an isotopy $H:S^3\times I\rightarrow C_{\sigma}\times I$ between $g$ and $g'$.}
\item[Claim 2.]{Denote $${\widetilde\Delta}^2:=\{x_1\leq -\frac{1}{\sqrt{2}}, x_4=x_5=x_6=x_7=0\}\subset S^6.$$

Then the algebraic number of points of intersection ${\rm Im}(H)\overset{\cdot}{\cap} ({\widetilde\Delta}^2\times I)$ is zero.}
\end{itemize}

	\begin{figure}[H]
	\centering
	\includegraphics[width=120mm]{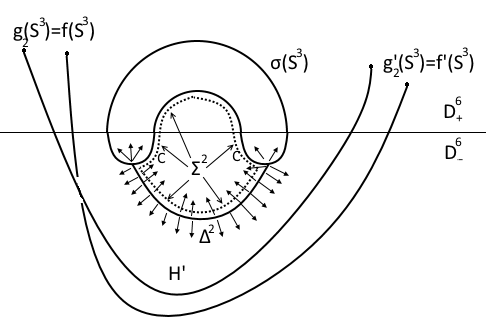}
	\caption{Proof of Lemma \ref{lemIsotopic}. We ``push'' ${\rm Im}(H')$ out of $D^6_-\times I$ along paths starting at ${\widetilde\Delta}^2\times I$.}
	\label{figIsotopy}
	\end{figure}

By Claim 2 and by the Whitney Trick Theorem, there is a concordance $H':S^3\times I\rightarrow C_{\sigma}\times I$ between $g$ and $g'$ whose image ${\rm Im}H'$ is disjoint from ${\widetilde\Delta}^2\times I$.

Now we ``push'' $H'(S^3\times I)$ out of $D^6_-\times I$ along paths starting at ${\widetilde\Delta}^2\times I$ analogously to how we ``pushed'' $f(S^3)$ out of $D^6_-$ along paths starting at $\Delta^2$ in the proof of Lemma \ref{lemExistStandardized}. The result of the ``pushing'' is a concordance $H'':S^3\times I\rightarrow C_{\sigma}\times I$ between $g$ and $g'$ whose image is in $D^6_+\times I$.

So there is a concordance $\std\times{\rm id} \sqcup H'':(S^2\times S^1 \sqcup S^3)\times I\rightarrow S^6\times I$ between $f$ and $f'$. Since in codimension at least $3$ concordance implies isotopy, it follows that $f$ is isotopic to $f'$.
\end{proof}

\begin{rem}
Lemma \ref{lemIsotopic} is similar to part (b) of Standartization Lemma 2.1 in \cite{Sk15}. The analogue of Claim 2 is used in the proof of part (b) of Standartization Lemma. Under the dimension restrictions of the hypothesis of Standartization Lemma, however, the analogue of Claim 2 holds just by general position.
\end{rem}

\begin{proof}[Proof of Claim 1.]
Consider linked spheres $\sigma\sqcup f|_{S^3}:S^3\sqcup S^3\rightarrow S^6$ and $\sigma\sqcup f'|_{S^3}:S^3\sqcup S^3\rightarrow S^6$. We have that 
$$\mu(\sigma\sqcup f|_{S^3}) \overset{(1)}{=} \widehat{\mu}(f) \overset{(2)}{=} \widehat{\mu}(f') \overset{(3)}{=} \mu(\sigma\sqcup f'|_{S^3}), \text{ where}$$
\begin{itemize}
\item[(1,3)]{hold by Lemma \ref{lemSameLNumbers}.}
\item[(2)]{holds by the hypothesis of the Lemma \ref{lemIsotopic}.}
\end{itemize}
Similarly, $\nu(\sigma\sqcup f|_{S^3})=\nu(\sigma\sqcup f'|_{S^3})$. 

By Theorem \ref{Hthm}, links $\sigma\sqcup f|_{S^3}$ and $\sigma\sqcup f'|_{S^3}$ are isotopic to $g_{m,n}$ and $g_{m',n'}$ respectively, for some integers $m,n,m'$, and $n'$, such that $m\equiv n\pmod{2}$ and $m'\equiv n'\pmod{2}$. By Lemma \ref{lemLambda}, we have that $m=\mu(\sigma\sqcup f|_{S^3})$, $n=\nu(\sigma\sqcup f|_{S^3})$, $m'=\mu(\sigma\sqcup f'|_{S^3})$, and $n'=\nu(\sigma\sqcup f'|_{S^3})$. Since $\mu(\sigma\sqcup f|_{S^3})=\mu(\sigma\sqcup f'|_{S^3})$ and $\nu(\sigma\sqcup f|_{S^3})=\nu(\sigma\sqcup f'|_{S^3})$, it follows that $m=m'$ and $n=n'$, i.e., that $\sigma\sqcup f|_{S^3}$ is isotopic to $\sigma\sqcup f'|_{S^3}$.

It now follows by Lemma \ref{lemHthmFixed}, that $f|_{S^3}$ is isotopic to $f'|_{S^3}$ in $C_{\sigma}$, i.e., that $g$ is isotopic to $g'$.
\end{proof}

\begin{proof}[Proof of Claim 2.]

We may assume that ${\rm Im}(H)$ and ${\widetilde\Delta}^2\times I$ are in general position.

Let $C\subset \sigma(S^3)\cap D^6_-$ be an embedded cylinder diffeomorphic to $S^1\times I$ such that $\partial C=\partial {\widetilde\Delta}^2 \sqcup \partial(\std(D^2_+\times 1))$. Let $[{\widetilde\Delta}^2]$, $[C]$, and $[\std(D^2_+\times 1)]$ be $2$-chains in $C_2(S^6;{\mathbb Z})$ with supports ${\widetilde\Delta}^2$, $C$, and $\std(D^2_+\times 1)$, respectively, and such that $\partial([{\widetilde\Delta}^2]+[C]+[\std(D^2_+\times 1)])=0$. Denote $\Sigma^2:=[{\widetilde\Delta}^2]+[C]+[\std(D^2_+\times 1)]$ (Fig.~\ref{figIsotopy}).

Clearly, 
\begin{multline*}
{\rm lk}({\rm Im}(H_0), \Sigma^2\times 0) = {\rm lk}(g_2(S^3), \Sigma^2) = {\rm lk}(f(S^3), \Sigma^2) = \varkappa(f) =\\ = \varkappa(f') = {\rm lk}({\rm Im}(H_1), \Sigma^2\times 1).
\end{multline*}

Hence, the algebraic number of points of the intersection ${\rm Im}(H)\overset{\cdot}{\cap} (\Sigma^2\times I)$ is zero. Also ${\rm Im}(H)\cap (C \cup \std(D^2_+\times 1))\times I=\emptyset$ since $C \cup \std(D^2_+\times 1)\subset \sigma(S^3)$. Therefore, the algebraic number of points of the intersection ${\rm Im}(H)\overset{\cdot}{\cap} ({\widetilde\Delta}^2\times I)$ is zero as well.

\end{proof}

\begin{proof}[Proof of Lemma \ref{lemChangeL1}]
It suffices to consider only the case $c=1$. Since $f$ is upper simple, there exists $\epsilon>0$ such that $f(S^3)\subset \{x_1>2\epsilon\}\subset D^6_+$.

Denote $\overline{\epsilon}:=\sqrt{1-\epsilon^2}$. Define an embedding $w:S^3\rightarrow S^6$ by the formula:
$$w(x_1,x_2,x_3,x_4):=(\epsilon,\overline{\epsilon}x_1,\overline{\epsilon}x_2,0,0,\overline{\epsilon}x_3,\overline{\epsilon}x_4).$$

Clearly, $w(S^3)\cap \std(S^2\times S^1)=\emptyset$. Since $w(S^3)\subset \{x_1=\epsilon\}$ and $f(S^3)\subset \{x_1>2\epsilon\}$, it follows that $w(S^3)\cap f(S^3)=\emptyset$.

Let $f':S^2\times S^1\sqcup S^3\rightarrow S^6$ be the upper simple embedding such that $f'|_{S^3}=f|_{S^3}\#w$. We shall prove that $f'$ is as required.

The sphere $w(S^3)$ bounds the disk $\{x_1\leq \epsilon, x_4=0, x_5=0\}$ in $S^6$. Clearly, this disk is disjoint from $\std(S^2\times S^1)$. This disk is also disjoint from $f(S^3)$, since it lies in $\{x_1\leq\epsilon\}$ and since $f(S^3)\subset \{x_1>2\epsilon\}$. Therefore there is an isotopy between $f'$ and $f$ fixed on $S^2\times S^1$. By Lemma \ref{lemFixedIsotopy}, it follows that $\widehat{\nu}(f')=\widehat{\nu}(f)$.

It remains to prove that $\widehat{\mu}(f')-\widehat{\mu}(f)=2\varkappa(f)$. 

The sphere $w(S^3)$ also bounds the disk $W:=\{x_1=\epsilon, x_4\geq 0, x_5=0\}$ in $S^6$. Consider the framed submanifold $W_{e_5, e_1}$. Its natural orientation agrees with the orientation of its boundary $w(S^3)$. By Lemma \ref{lemExistsDisk}, there is a $4$-disk $\Delta\subset \{x_1>2\epsilon\}$ with boundary $f(S^3)$. The disk $\Delta$ is disjoint from $W$ since $W\subset \{x_1=\epsilon\}$. Frame $\Delta$ arbitrarily so that its natural orientation agrees with the orientation of its boundary $f(S^3)$.

Denote $a:=\std^{-1}|_{S^2\times S^1}(W_{e_5, e_1})$ and $b:=\std^{-1}|_{S^2\times S^1}(\Delta)$. Then $a$ and $b$ are framed $1$-submanifolds of $S^2\times S^1$. We have 
\begin{multline*}
\widehat{\mu}(f')-\widehat{\mu}(f) = \hp(\theta(a\sqcup b)) - \hp(\theta(b)) =\\= \hp(\theta(a))+\hp(\theta(b))+2{\rm lk}(\theta(a),\theta(b)) - \hp(\theta(b)) =\\= \hp(\theta(a))+2{\rm lk}(\theta(a),\theta(b)) = 0 + 2\varkappa(f).
\end{multline*}
It remains to prove the last equation, i.e., that $\hp(\theta(a))=0$ and ${\rm lk}(\theta(a),\theta(b))=\varkappa(f)$.

The framed intersection $W_{e_5, e_1}\cap\std(S^2\times S^1)$ is the framed circle $$\{x_1=\epsilon,\ x_2^2+x_3^2=\frac{1}{2}-\epsilon^2,\ x_4=\frac{1}{\sqrt{2}}\}_{e_5, e_1}\subset S^6.$$ So, $a=\std^{-1}|_{S^2\times S^1}(W_{e_5, e_1})$ is the framed circle $$(\{x_1=\sqrt{2}\epsilon, x_2^2+x_3^2=1-2\epsilon^2\}\times 1)_{0\times e_2,e_1\times 0}\subset S^2\times S^1.$$ Then $\theta(a)$ is the framed circle $$\{x_1^2+x_2^2=\frac{1}{2}-\epsilon^2,x_3=\sqrt{\frac{1}{2}+\epsilon^2},x_4=0\}_{e_4,v}\subset S^3,$$ for a certain normal field $v$. Clearly, $\hp(\theta(a))=0$.

Consider the handlebody $T:=\{x_1^2+x_2^2\leq\frac{1}{2}, x_3^2+x_4^2\geq\frac{1}{2}\}$ in $S^3$ diffeomorphic to $D^2\times S^1$. The circle $\theta(a)$ is ``almost'' the meridian $S^1\times 1$ of $T$. The circle $\theta(b)$ lies in ${\rm Int}T$ and is ``further away'' from $\partial T$ than $\theta(a)$. Moreover, $\theta(b)$ is homologous to $\varkappa(f)[0\times S^1]$ in $T$. Therefore, ${\rm lk}(\theta(a),\theta(b))=\varkappa(f)$.

\end{proof}

To prove Lemma \ref{lemChangeL2}, we need the following Lemma \ref{lem2Isotopy}.

\begin{lem2Isotopy}
\label{lem2Isotopy}
There exists an isotopy $H:S^6\times I \rightarrow S^6\times I$ such that $H_1|_{\std(S^2\times S^1)}={\rm id}$, $H_1(C_{\partial\std})=C_{\partial\std}$, and $[\widetilde{H} z_3] = 2[z_2]+[z_3]$, where $\widetilde{H}:C_{\partial\std}\rightarrow C_{\partial\std}$ is the abbreviation of $H_1$.
\end{lem2Isotopy}

\begin{proof}
Define an embedding $z_4:S^4\rightarrow S^6$ by the formula
$$z_4(x_1,x_2,x_3,x_4,x_5):=(x_1,x_2,x_3,0,0,x_4,x_5).$$
Note that $z_4(S^4)$ is a sphere isotopic to the standard sphere and linked with $\std(0\times S^1)$.

Frame $\std(D^3\times S^1)$ with vectors $e_6,e_7$ at each point (this is the standard framing). Frame $z_4(S^4)$ with vectors $e_4,e_5$ at each point. By Lemma \ref{lemFramedSum}, we may extend the framings of $\std((D^3\times S^1)_0)_{e_6,e_7}$ and $z_4(S^4_0)_{e_4,e_5}$ to a framing $\zeta$ of $(\std\#z_4)(D^3\times S^1)$.

We have framed the images of the embeddings $\std\#z_4:D^3\times S^1\rightarrow S^6$ and $\std:D^3\times S^1\rightarrow S^6$. Take a $3$-disk $B^3\subset D^3$ such that $B^3\times S^1\subset (D^3\times S^1)_0$. Clearly, $(\std\#z_4)|_{B^3\times S^1}=\std|_{B^3\times S^1}$ and the framings of $(\std\#z_4)(B^3\times S^1)$ and $\std(B^3\times S^1)$ are the same.

So, by Lemma \ref{lemFramedHandlebodyIsotopy}, there is a framed isotopy $H:S^6\times I \rightarrow S^6\times I$ between the embeddings $\std\#z_4$ and $\std$ (the isotopy $H$ is not assumed to be fixed on the boundary $\partial(D^3\times S^1)$). We may assume that $H_1(C_{\partial\std})=C_{\partial\std}$. We shall prove that $H$ is as required.

Lemmas \ref{lemBasis} and \ref{lemTorIsotopy} assert that $$[\widetilde{H} z_3]=\hp(p_2\widetilde{H} z_3)[z_2]+{\rm deg}(p_3\widetilde{H} z_3)[z_3] \text{\quad and\quad} [\widetilde{H} z_3]=2k[z_2]+[z_3],\text{\quad for some integer }k.$$ By Lemma \ref{lemBasis}, the homotopy classes $[z_2]$ and $[z_3]$ form a basis of $\pi_3(C_{\partial\std})$, so ${\rm deg}(p_3\widetilde{H} z_3)=1$. So $[\widetilde{H} z_3] = \hp(p_2\widetilde{H} z_3)[z_2]+[z_3]$. It remains to prove that $\hp(p_2\widetilde{H} z_3)=2$.

By the definition of $p_2$, the framed submanifold $\std(D^3\times S^1)_{e_6,e_7}\cap C_{\partial\std}$ is the framed preimage of a regular point of $p_2:C_{\partial\std}\rightarrow S^2$. Therefore, the framed submanifold $(\std\#z_4)(D^3\times S^1)_\zeta\cap C_{\partial\std}$ is the framed preimage of a regular point of $p_2\widetilde{H}:C_{\partial\std}\rightarrow S^2$.

By general position, we may assume that $$z_3(S^3)\cap (\std\#z_4)(D^3\times S^1) = (z_3(S^3)\cap \std(D^3\times S^1)) \sqcup (z_3(S^3)\cap z_4(S^4)).$$ So, the value of the Hopf invariant $\hp(p_2\widetilde{H} z_3)$ is equal to the value of the Hopf invariant of the union of framed circles
\begin{multline*}
z^{-1}_3(\std\#z_4(D^3\times S^1)_\zeta)=z^{-1}_3(\std(D^3\times S^1)_{e_6,e_7})\sqcup z^{-1}_3(z_4(S^4)_{e_4,e_5}) =\\= \{x_1^2+x_2^2=1\}_{e_3,e_4} \sqcup \{x_3^2+x_4^2=1\}_{e_1,e_2} \subset S^3.
\end{multline*}
In other words, 
\begin{multline*}
\hp(p_2\widetilde{H} z_3)=h(\{x_1^2+x_2^2=1\}_{e_3,e_4} \sqcup \{x_3^2+x_4^2=1\}_{e_1,e_2})=\\=h(\{x_1^2+x_2^2=1\}_{e_3,e_4}) + h(\{x_3^2+x_4^2=1\}_{e_1,e_2}) + 2{\rm lk}(\{x_1^2+x_2^2=1\}, \{x_3^2+x_4^2=1\})=\\=0+0+2=2.
\end{multline*}
\end{proof}

\begin{proof}[Proof of Lemma \ref{lemChangeL2}]
It suffices to consider only the case $c=1$.

Lemma \ref{lem2Isotopy} asserts that there is an isotopy $H:S^6\times I \rightarrow S^6\times I$ such that $H_1|_{\std(S^2\times S^1)}={\rm id}$, $H_1(C_{\partial\std})=C_{\partial\std}$, and $[\widetilde{H} z_3]= 2[z_2]+[z_3]$, where $\widetilde{H}:C_{\partial\std}\rightarrow C_{\partial\std}$ is the abbreviation of $H_1$.

We shall prove that the simple embedding $f':=H_1f$ is as required, i.e, that $\widehat{\nu}(f')=\widehat{\nu}(f)+2\varkappa(f)$. Let $\widetilde{H}^*:\pi_3(C_{\partial\std})\rightarrow \pi_3(C_{\partial\std})$ be the homomorphism induced by $\widetilde{H}$. We have
\begin{multline*}
\widehat{\nu}(f')[z_2]+\varkappa(f')[z_3] \overset{(1)}{=} [{\delet f}'] = [\widetilde{H} {\delet f}] = \widetilde{H}^*[{\delet f}] \overset{(2)}{=} \widetilde{H}^*(\widehat{\nu}(f)[z_2]+\varkappa(f)[z_3]) =\\= \widehat{\nu}(f)[\widetilde{H} z_2]+\varkappa(f)[\widetilde{H} z_3] \overset{(3)}{=} (\widehat{\nu}(f)+2\varkappa(f))[z_2] + \varkappa(f)[z_3], \text{ where}
\end{multline*}
\begin{itemize}
\item[(1,2)]{hold by Lemma \ref{lemInBasis}.}
\item[(3)]{holds because $[\widetilde{H} z_2] = [z_2]$ by Part (I) of Lemma \ref{lemTorIsotopy}, and $[\widetilde{H} z_3] = 2[z_2]+[z_3]$ by the choice of $H$.}
\end{itemize}
By Lemma \ref{lemBasis}, the homotopy classes $[z_2]$ and $[z_3]$ form a basis of $\pi_3(C_{\partial\std})$, so $\widehat{\nu}(f')=\widehat{\nu}(f)+2\varkappa(f)$.

\end{proof}

\medskip
\section{Proof of Lemma \ref{lemExamples}.}
To prove Lemma \ref{lemExamples}, we need the following Lemma \ref{lemAdd}.

\begin{lemAdd}
\label{lemAdd}
Let $f$ be a upper simple embedding. Let $g:S^3_1\sqcup S^3_2\rightarrow S^6$ be a Brunnian embedding. Suppose that the image of $g$ lies in a small $6$-disk in ${\rm Int}D^6_+$ disjoint from the image of $f$ and $\std(D^3\times S^1)$. Then:
\begin{itemize}
\item[(I)]{$\varkappa(f\#g)=\varkappa(f)$}
\item[(II)]{$\mu(f\#g)=\mu(f)+\rho_{0,2\varkappa(f)}(\mu(g))$}
\item[(III)]{$\nu(f\#g)=\nu(f)+\rho_{0,2\varkappa(f)}(\nu(g))$}
\end{itemize}

\end{lemAdd}

\begin{proof}
Part (I) is obvious.

The following construction is needed to prove parts (II) and (III).

Let $\Delta_f\subset {\rm Int}D^6_+$ be an embedded $4$-disk such that $\partial \Delta_f=f(S^3)$. Since $g$ is a simple embedding, there are embedded $4$-disks $\Delta_{1},\Delta_{2}\subset {\rm Int}D^6_+$ such that $\partial \Delta_{1}=g(S_1^3)$ and $\partial \Delta_{2}=g(S_2^3)$. We may assume that $\Delta_{1}$ and $\Delta_{2}$ lie in a small $6$-disk in ${\rm Int}D^6_+$ disjoint from $\Delta_f$ and $\std(D^3\times S^1)$.

The definition of the {\it boundary embedded connected sum} $M\sharp N\subset S^k$ of the embedded compact orientable manifolds $M\subset S^k$ and $N\subset S^k$ with boundary is analogous to the definition of the (absolute) embedded connected sum. Suppose that $M$ and $N$ are framed. Then the framing of $M\sqcup N$ can be extended to a framing of $M\sharp N$ (analogous to Lemma \ref{lemFramedSum}).

Frame $\Delta_f$, $\Delta_{1}$, and $\Delta_{2}$ arbitrarily so that their natural orientations agree with the orientations of their respective boundaries $f(S^3)$, $g(S_1^3)$, and $g(S_2^3)$. Extend the framings of $\std(D^3\times S^1)_{e_6,e_7}$, $\Delta_f$, $\Delta_{1}$, and $\Delta_{2}$ to framings of $\std(D^3\times S^1)\sharp \Delta_{1}$ and $\Delta_f\sharp \Delta_{2}$.

By Lemma \ref{lemFramedHandlebodyIsotopy}, there is a framed isotopy $H:S^6\times I\rightarrow S^6\times I$ between the framed submanifolds $\std(D^3\times S^1)\sharp \Delta_{1}$ and $\std(D^3\times S^1)_{e_6,e_7}$.

Then $H_1(f\#g):S^2\times S^1\sqcup S^3\rightarrow S^6$ is a upper simple embedding and $H_1(\std(D^3\times S^1)\sharp \Delta_{1})=\std(D^3\times S^1)_{e_6,e_7}$.

{\it Completion of the proof of} (II).
By the definition of $\mu$
\begin{multline*}
\mu(f\#g) \equiv
\mu(H_1(f\#g)) \equiv
\hp(\theta\std^{-1}|_{S^2\times S^1}(H_1(\Delta_f\sharp\Delta_{2}))) \equiv \\
\equiv \hp(\theta(\std|_{S^2\times S^1}\#g|_{S^3_1})^{-1}(\Delta_f\sharp\Delta_{2}))
\equiv \hp(\theta(\std|_{S^2\times S^1}\#g|_{S^3_1})^{-1}(\Delta_f)\sqcup \theta(\std|_{S^2\times S^1}\#g|_{S^3_1})^{-1}(\Delta_{2})) \overset{(1)}{\equiv} \\
\overset{(1)}{\equiv} \hp(\theta(\std|_{S^2\times S^1}\#g|_{S^3_1})^{-1}(\Delta_f)) + \hp(\theta(\std|_{S^2\times S^1}\#g|_{S^3_1})^{-1}(\Delta_{2}))
\equiv \mu(f)+\mu(g) \pmod{2\varkappa(f)}.
\end{multline*}

Here (1) holds since framed $1$-submanifolds $\theta(\std|_{S^2\times S^1}\#g|_{S^3_1})^{-1}(\Delta_f)$ and $\theta(\std|_{S^2\times S^1}\#g|_{S^3_1})^{-1}(\Delta_{2})$ of $S^3$ are unlinked (they lie in disjoint $3$-disks).

{\it Completion of the proof of} (III).
By the definition of $\nu$
\begin{multline*}
\nu(f\#g) \equiv
\nu(H_1(f\#g)) \equiv
\hp({(H_1{\delet f}\#H_1g|_{S^3_2})}^{-1}\std(D^3\times S^1)) \overset{(1)}{\equiv}\\
\overset{(1)}{\equiv} \hp({(H_1{\delet f})}^{-1}\std(D^3\times S^1)) + \hp({(H_1g|_{S^3_2})}^{-1}\std(D^3\times S^1)) \equiv \\
\equiv \hp({\delet {f}}^{-1}\std(D^3\times S^1)) + \hp({g|_{S^3_2}}^{-1}(\Delta_{1})) \equiv \nu(f) + \nu(g) \pmod{2\varkappa(f)}.
\end{multline*}

Here (1) holds since framed $1$-submanifolds ${(H_1{\delet f})}^{-1}\std(D^3\times S^1)$ and ${(H_1g|_{S^3_2})}^{-1}\std(D^3\times S^1)$ of $S^3$ are unlinked (they lie in disjoint $3$-disks).

\end{proof}

\begin{proof}[Proof of Lemma \ref{lemExamples}]
By Lemma \ref{lemAdd}, it suffices to consider only the case $m=n=0$.

Clearly, $\varkappa(f_{k,0,0})={\rm lk}(f_{k,0,0}(S^2\times 1), f_{k,0,0}(S^3))=k{\rm lk}(\std(S^2\times 1), z_3(S^3))=k$.

Suppose that $k=0$. By definition of $f_{0,0,0}$, the submanifolds $f_{0,0,0}(S^2\times S^1)=\std(S^2\times S^1)\subset S^6$ and $f_{0,0,0}(S^3)\subset S^6$ lie in disjoint $6$-disk. So, there is $4$-disk bounded by $f_{0,0,0}(S^3)$ and disjoint from $f_{0,0,0}(S^2\times S^1)$, meaning that $\mu(f_{0,0,0})=0$. Also, we may assume that $f_{0,0,0}(S^3)$ is disjoint from $\std(D^3\times S^1)$, meaning that $\nu(f_{0,0,0})=0$.

Suppose now that $k\neq 0$. It suffices to consider only the case $k=1$.

Consider the $4$-disk $$\Delta:=\{x_1\geq \frac{1}{10}, x_2=x_3=0 \}\subset S^6.$$
Clearly, $\Delta\subset{\rm Int}D^6_+$ and $\partial\Delta = z_{3,\frac{1}{10}}(S^3) = f_{1,0,0}(S^3)$.
We have
$$\widehat{\mu}(f_{1,0,0}) = \hp\theta\std^{-1}|_{S^2\times S^1}(\Delta_{e_2, e_3}) = \hp(\{x_3^2+x_4^2=1\}_{e_1,e_2}) = 0.$$
Hence, $\mu(f_{1,0,0})=0$.

The map ${\delet f}_{1,0,0}$ is homotopic to $z_3$. Therefore, by Lemmas \ref{lemBasis} and \ref{lemInBasis} it follows that $\widehat{\nu}(f_{1,0,0})=0$, so $\nu(f_{1,0,0})=0$.
\end{proof}

{}

\end{document}